\documentclass[11pt]{article}
\usepackage{bm, mathrsfs, graphics,amssymb,amsmath,subeqnarray,setspace,graphicx,amsthm,epstopdf,subfigure,color,float,hyperref,authblk}
\usepackage{enumerate,algorithmic,algorithm}
\usepackage{mathtools}

\usepackage{upref}
\usepackage{svg}
\usepackage{geometry}
\usepackage{epstopdf}

\def\R{\mathbb{R}}

\def\md{\mathop{}\mathopen\mathrm{d}}

\newcommand{\tcb}[1]{\textcolor{blue}{#1}}

\newtheorem{theorem}{Theorem}[section]
\newtheorem{lemma}{Lemma}[section]

\newtheorem{remark}{Remark}[section]

\newtheorem{assumption}{Assumption}[section]

\DeclareRobustCommand{\corrauthormark}{\footnotemark[4]}

\numberwithin{equation}{section}

\title{A particle method for the Boltzmann equation via amortized sampling from
Green's function of the lifted linear operator}

\author[a]{Zichang Ju\thanks{E-mail: jzc\_sypqf@sjtu.edu.cn}}
\author[a,b]{
Lei Li\thanks{Email: leili2010@sjtu.edu.cn}
\corrauthormark
}
\author[a]{
Yang Yu 
\thanks{Email: yuyang2357@sjtu.edu.cn}
\thanks{Corresponding authors}
}

\affil[a]{School of Mathematical Sciences, Shanghai Jiao Tong University, Shanghai 200240, China}

\affil[b]{Institute of Natural Sciences, MOE-LSC, Shanghai Jiao Tong University, Shanghai 200240, China}

\date{}

\begin{document}

\maketitle

% \footnotetext[4]{Corresponding authors.}

\begin{abstract}
The collision operator for the Boltzmann equation is a nonlinear nonlocal operator. When lifted in the extended 2-particle space, it is viewed as the projection of a collisional linear operator.
In this paper, we propose a particle method that samples the post-collision relative velocity directly from the Green's function of this operator (the transition probability of the generated time-continuous Markov chain). The normalizing flow amortized sampling is then proposed to reduce the sampling complexity.  The resulted method takes $O(N)$ each time where $N$ is the particle number, and conserves momentum and energy exactly.  This method does not require the boundedness of the kernel and, more importantly, it allows learning the Green's function directly from the scattering data without selecting the kernel in a specified family.
\end{abstract}

\section{Introduction}\label{sec:intro}
The Boltzmann equation is one of the fundamental models in kinetic theory for describing the evolution of the one-particle distribution function $f(t,x,v)$ of a dilute gas away from thermodynamic equilibrium \cite{cercignani1988boltzmann, cercignani2013mathematical}, where $t$ denotes time, $v=(v_1,\cdots,v_d)\in \mathbb{R}^d$ denotes the velocity, $x=(x_1,\cdots,x_d)\in \Omega\subset \mathbb{R}^d$ denotes the particle position, and $\Omega$ is the spatial domain. It provides a mesoscopic description between microscopic particle dynamics and macroscopic fluid models, and plays an important role in rarefied and nonequilibrium gas dynamics. Representative applications arise in high-altitude and hypersonic flows, molecular gas dynamics, and gas transport at micro- and nanoscale dimensions \cite{bird1994molecular,boyd2017nonequilibrium,karniadakis2005microflows,sone2007molecular}. In these regimes, molecular collisions and their associated scattering laws have a direct influence on the evolution of the velocity distribution.

In the following we focus on the 3-dimensional case, for which the Boltzmann equation reads:
\begin{equation}
    \partial_t f+v\cdot \nabla_x f=
    Q(f,f),
\end{equation}
where $Q(f,f)$ is a quadratic operator with the form
\begin{equation}
    Q(f,f)(v):=\int_{\mathbb{R}^{3}}\int_{\mathbb{S}^{2}} B(v-v_*,  n')(f(v_*^{\prime})f(v^{\prime})-f(v_*)f(v)) \mathrm{d}n{'} \mathrm{d}v_*.
\end{equation}
Here for simplicity the variables $t$ and $x$ are omitted. Moreover, $v$ and $v_*$ are the velocities of pre-collision particles while $v' $ and $v'_*$ are the corresponding velocities of post-collision particles related by the elastic collision law: 
\begin{equation}
    v' = \bar v+|q| n'/2, \qquad   v'_*=  \bar v - |q| n'/2,
\end{equation} 
where we denote 
\begin{equation}
\bar v=(v+v_*)/2, \quad q=v-v_*,\quad  n=q/{|q|}
\end{equation}
for simplicity. Under the rotational invariance assumption on the collision law, the collision kernel $B(q,n')$ has the form $B(q,n') = B(|q|, \cos \theta)$, where $\cos\theta = n\cdot n'$. We use the  rotational invariance assumption all through the work. Moreover, one has $B(|q|,\cos\theta)=|q|\sigma(|q|,\theta)$
%Denote $\cos\theta=\mu=n \cdot n' $, and set $B(|q|,\cos\theta)=|q|\sigma(|q|,\theta)$ the collision kernel for a central potential
, where $\sigma(|q|,\theta)$ is the differential collision cross section. For cutoff kernels, the total cross section is given by 
\begin{equation}
\label{eq:sigmatotal}
    %\sigma_{\text{total}}(|q|)=2\pi\int_{[0,\pi]}\sigma(|q|,\theta)\sin\theta\,\mathrm{d}\theta ,
    \sigma_{\text{total}}(|q|)=\int_{\mathbb{S}^{2}}\sigma(|q|,\theta)\mathrm{d}n' =2\pi\int_0^{\pi}\sigma(|q|,\theta)\sin\theta\,\mathrm{d}\theta.
\end{equation}

The Boltzmann collision equation has several fundamental physical structures \cite{cercignani1988boltzmann}. It preserves
the collision invariants corresponding to the conservation of mass, momentum and kinetic energy. It also satisfies the entropy dissipation property, usually referred to as Boltzmann's \(H\)-theorem.
The numerical treatment of the Boltzmann equation is challenging because the collision operator is nonlinear and nonlocal and involves a high-dimensional integration over the collision partners and scattering directions.

%The mathematical well-posedness theory of the Boltzmann equation is delicate because the collision operator is nonlinear, nonlocal, and may be singular for non-cutoff kernels. For general large initial data, a fundamental framework is provided by the DiPerna-Lions theory of renormalized solutions, which gives global existence and weak stability under broad assumptions \cite{dipernaLions1989}. More classical or strong well-posedness results areusually obtained under additional assumptions, such as cutoff kernels, hard potentials, or perturbations of equilibrium. For non-cutoff kernels, the angular singularity produces additional regularizing effects and leads to fractional diffusion-type behavior in velocity \cite{alexandre2000entropy,villani2002review}.

To solve the Boltzmann equation accurately and efficiently, both deterministic and particle methods have been extensively developed. Among deterministic approaches, discrete velocity methods \cite{mouhot2013convolutive} and spectral methods \cite{pareschi2000numerical} are two representative classes for the numerical treatment of the Boltzmann collision operator, where the latter can achieve spectral accuracy for sufficiently smooth solutions \cite{pareschi2000numerical} and substantially reduce the computational cost of collision evaluation \cite{gamba2017fast}. Despite these advantages in accuracy and computational efficiency, spectral methods do not automatically preserve positivity or the exact conservation of momentum and energy, and suitable modifications may therefore be required \cite{filbet2011analysis,pareschi2022moment}. 
%Deterministic methods including discrete velocity methods (DVM) \cite{bobylev1995approximation,goldstein1989investigations,panferov2002new,rogier1994direct}, Fourier-Galerkin method \cite{pareschi1996fourier}, 
%finite-volume or finite-difference discretizations in phase space.
%Spectral methods are attractive because they can achieve spectral accuracy for smooth solutions, but the preservation of positivity and physical invariants often requires additional care. In contrast, particle methods \cite{HOMOLLE20072341,rjasanow1996stochastic} approximate the distribution function by an ensemble of particles and simulate transport and collisions statistically. %This class is especially important for high-dimensional and rarefied-gas regimes.
In contrast, particle methods \cite{bird1994molecular,bird2013dsmc,pareschi2001introduction,rjasanow1996stochastic} represent the distribution function by an ensemble of particles and model the collisional dynamics statistically, and are able to capture some microscopic behaviors that the reduced PDE models may have lost.  Among them, the Direct Simulation Monte Carlo (DSMC) method \cite{bird1994molecular,bird2013dsmc} is one of the most widely used approaches for the Boltzmann equation, particularly in rarefied-gas dynamics. In DSMC and related Monte Carlo schemes, the distribution function is approximated by an empirical particle distribution, while the collision operator is realized through randomly selected particle pairs and sampled post-collision states according to the prescribed collision law \cite{nanbu1980direct,pareschi2001introduction}. By avoiding a deterministic discretization of the velocity space, DSMC circumvents the high-dimensional velocity grids required by deterministic kinetic solvers. For elastic binary collisions, momentum and energy can be preserved at the level of each accepted collision pair through the collision parametrization \cite{bird1994molecular}. On the other hand, the stochastic nature of these methods introduces statistical fluctuations, so that accurate estimation of the distribution function or its moments may require a sufficiently large number of simulation particles \cite{HOMOLLE20072341}.

A key ingredient in DSMC is the specification of a molecular collision model, which determines both the collision rate and the distribution of post-collision scattering directions \cite{bird1994molecular}. In practical simulations, commonly used models such as the hard sphere, variable hard sphere (VHS), and variable soft sphere (VSS) models \cite{bird1994molecular,koura1991variable} prescribe parametric forms of the collision cross section and angular scattering law, thereby providing explicit rules for collision sampling. More generally, DSMC is not restricted to these parametric models; rather, its implementation requires sufficient information to evaluate collision rates and sample post-collision states. When such scattering information is available only through a finite set of experimental measurements \cite{rothe1959total,von2015analysis}, a convenient collision model or sampling rule may not be directly available. This provides a motivation for extracting the relevant collision information from data and incorporating it into particle collision dynamics without prescribing a specific parametric molecular model.

Data-driven methods have recently been introduced into kinetic modeling, including neural surrogate models for the Boltzmann collision operator $f\mapsto Q(f,f)$ \cite{xiao2021using,xiao2023relaxnet}, data-driven reduced-order approximations for accelerating collision evaluation \cite{alekseenko2022fast}, and inverse approaches for recovering collision kernels from indirect measurements \cite{lai2021reconstruction,li2023determining}. More recently, machine-learned microscopic collision models have begun to be incorporated directly into DSMC to replace prescribed phenomenological scattering laws or expensive high-fidelity collision calculations \cite{roohi2026physics}.

In this paper, We first develop the collision algorithm in the spatially homogeneous setting, where the collision mechanism can be isolated from spatial transport. The same collision step subsequently can be incorporated into the particle-in-cell (PIC) framework with operator splitting \cite{bailo2024collisional, serikov1999particle} for spatially non-homogeneous problems. Compared with researches before, the present work addresses a different learning problem. We do not seek to approximate the nonlinear mapping $f\mapsto Q(f,f)$ or select and calibrate a classical parametric collision model. Instead, we lift the binary collision dynamics to the two-particle space and learn the Green's function (a probability distribution). For fixed center-of-mass velocity and relative speed, the lifted collision action reduces to a linear angular operator $\Omega_J$ acting on the relative direction $n\in \mathbb S^2$. Under rotational invariance, spherical harmonics diagonalize this operator,
\begin{equation*}
    \Omega_JY_{\ell m} = -\lambda_\ell(|q|)Y_{\ell m},
\end{equation*}
so that the sequence $\{\lambda_\ell(|q|)\}_{\ell \geq 0}$ determines the Green’s function of the corresponding angular evolution. This representation also separates the information contained in two types of scattering measurements. The total cross section determines the overall collision rate
\begin{equation*}
    K(|q|) = |q| \sigma_{\text{total}}(|q|),
\end{equation*}
whereas samples of the scattering angle conditional on the relative speed determine the normalized angular shape. We fit these quantities through an impact-parameter representation and then compute the spectral coefficients $\lambda_\ell (|q|)$. This avoids assuming an HS, VHS, or VSS collision law and avoids the potentially unstable division by $\sin \theta$ involved in explicitly reconstructing the differential cross section. For the data-to-spectrum reconstruction, we further assume that the scattering law admits a single-valued monotone relation between the impact parameter and the scattering angle. The resulting Green’s function is the finite-time transition law of the lifted angular process for a fixed particle pair. We decompose it into a no-collision atom and a continuous component, approximate the latter by a truncated spherical-harmonic expansion, and train a conditional monotone normalizing flow to amortize its repeated sampling over $(|q|, \Delta t)$. After the offline training stage, the collision update has $O(N)$ cost per time step. Since the center-of-mass velocity and the magnitude of the relative velocity are unchanged in every pairwise update, momentum and kinetic energy are preserved exactly.

The main methodological and analytical contributions of this work are summarized as follows.
\begin{enumerate}
    \item \textbf{A Green's function based particle collision rule.}
    We propose a particle collision rule that samples the finite-time angular transition generated by the lifted two-particle operator for each randomly paired particle pair. Rather than using only an eventwise single-collision angular update, the method incorporates the finite-time evolution of the lifted angular process into one collision step. After the offline construction of the sampler, the online collision stage has $O(N)$ complexity per time step and preserves momentum and kinetic energy exactly at the pair level.

    \item \textbf{Green's function reconstruction and scattering angle sampler.}
    We devise a procedure that infers the collision rate and the solver-relevant angular spectrum from scattering angle samples and total cross section measurements. The total cross section data determine the impact parameter scale and the collision rate, whereas the conditional scattering angle samples determine the normalized angular shape. These quantities are then used to compute the eigenvalues $\lambda_\ell(|q|)$ without prescribing an collision model. Then we construct a conditional normalizing flow tailored to the continuous Green's function component, thereby obtaining a sampler of the Green's function.

    \item \textbf{Convergence and nonnegativity of the truncated Green's function representation.}
    For rotationally invariant kernels with finite total cross section, we establish $L^2$ convergence of the truncated continuous component of the Green's function and $L^1$ convergence of its normalized density under $B(\cdot;|q|)\in L^2([-1,1])$. Under the additional assumption that the Legendre projections of $B$ converge uniformly, we derive an $L^\infty$ truncation estimate and prove that sufficiently high-order truncations are nonnegative uniformly for $\Delta t\in[\delta,T]$, with $\delta>0$. These results provide the analytical justification for using the finite spectral approximation as a probability density.

\end{enumerate}

The rest of this paper is organized as follows. Section \ref{sec:theoretical foundation} introduces the lifted collision operator, the Green’s function particle update, and its spherical harmonic representation. Section \ref{sec:Algorithm} presents the data-to-spectrum reconstruction, the conditional normalizing flow sampler, and the complete Boltzmann solver. Section \ref{sec:truncation} analyzes the convergence and positivity of the truncated Green’s function. Section \ref{sec:discom} discusses the formal large-particle consistency of the method and its relation to DSMC. Section \ref{sec:num} reports numerical experiments. %, and Section \ref{sec:conc} concludes the paper.

\section{The methodology and foundations}\label{sec:theoretical foundation}

We explain the main idea in subsection~\ref{subsec:linearlift} and make comments on related works, and gather some facts about the Green's function in subsection~\ref{subsec:green}.

\subsection{The methodology based on the linear lifting of the collision operator}\label{subsec:linearlift}

Our goal is to propose a particle method for the Boltzmann equation following our earlier work on Landau equation in \cite{du2025structure}. 
The intrinsic idea in \cite{du2025structure} is that the statistical effect of the post-collision relative velocity in the grazing regime is a spherical diffusion and can be sampled easily. Let us here briefly explain the fact from the nonlinear Landau operator
\[
\begin{split}
Q_L(f)&=
\frac{1}{2}\nabla_{v}\cdot\bigg( \int_{\mathbb{R}^d}
    A(v-v_{\ast})(f(v_{\ast})\nabla_{v}f(v)-f(v)\nabla_{v_\ast}f(v_\ast))\md v_\ast\bigg)\\
&=\frac{1}{2}
\int_{\mathbb{R}^d} (\nabla_v-\nabla_{v_\ast})\cdot(A(v-v_{\ast})
(\nabla_v-\nabla_{v_\ast})f(v)f(v_\ast))\mathrm d v_\ast.
\end{split}
\]
If we discard the integration, then one comes at the linear operator
\[
A_LF:=\frac{1}{2}(\nabla_v-\nabla_{v_\ast})\cdot(A(v-v_{\ast})
(\nabla_v-\nabla_{v_\ast})F),
\]
lifted in $(v, v_*)$ space recently used in \cite{guillen2025landau} for theoretic study.
With the notation $q$ before one may find that 
\[
\nabla_v-\nabla_{v_*}=2\nabla_q
\]
and thus $A_LF$ is exactly the spherical diffusion on the relative velocity agreeing the particle method in \cite{du2025structure}.
% \begin{remark}
% In fact, applying the large particle limit of the random batch type particle system \cite{jin2022mean}
% to the particle method in \cite{du2025structure}, one may easily recover this lifted linear operator for the two-particle system. 
% \end{remark}

Motivated by the above understanding of the Landau equation, it is natural to unwrap the nonlinear collision operator in the Boltzmann equation into a linear operator in $(v, v_*)$ space as
\begin{equation*}
A_BF=J F:=
\int_{\mathbb{S}^{2}} B(|v-v_*|, \frac{(v-v_*)\cdot n'}{|v-v_*|})(F(t,v^{\prime}, v_*^{\prime})-F(t,v, v_*)) \mathrm{d}n'.
\end{equation*}
With the molecular chaos $F(t, v, v_*)\approx f(t, v)f(t, v_*)$, one clearly has
$Q(f,f)=\int JF dv_*$.
Denote $n' = \frac{v'-v_*'}{|q|} $, $ \mu = n\cdot n'=\cos\theta$. Then the momentum and energy conservations give:
\begin{equation}
    v = \bar{v}+\frac{|q|}{2}n , \quad v_* = \bar{v}-\frac{|q|}{2}n , \quad
    v' = \bar{v}+\frac{|q|}{2}n' , \quad v_* '= \bar{v}-\frac{|q|}{2}n'.
    \label{eq:updatev}
\end{equation}
The operator can be written in a clean form as
\begin{equation}\label{eq:linearlift}
    J F(n,t;\bar{v}, |q|) :=\int_{\mathbb{S}^2} B(|q|,\mu)\left( F(n', t; \bar{v}, |q|) - F(n, t; \bar{v}, |q|) \right)\mathrm{d}n'.
\end{equation}

The operator $J$ is then a generator $\Omega_J$ of a jump process on the unit sphere (the operator $J$ acts on functions defined in $\R^{6}$ while $\Omega_J$ acts on functions defined on $\mathbb{S}^2$ for given $\bar{v}$ and $|q|$ determined by $v$ and $v_*$ pair). We will make use of the Green's function or semigroup of $\Omega_J$ to design our particle method.
Let %$G(n', t ; \bar{v}, |q|, n)$
$G(n', t ; |q|, n)$ be the Green's function, i.e., the solution to the linear equation
%\[
%\partial_t G=JG, \quad G(\cdot, 0; %\bar{v}, |q|, n)=\delta_{n}(\cdot).
%\]
\[
\partial_t G=\Omega_J G, \quad G(\cdot, 0;  |q|, n)=\delta_{n}(\cdot).
\]
Here, $\delta$ indicates the Dirac mass with respect to the Hausdorff measure on the unit sphere. When $|q|=0$, we let
$G(\cdot, t; |q|, n)\equiv \delta_n(\cdot)$ (this is not important). With the Green's function, one then has
\[
(e^{tJ}F_0)(v, v_*)=\int_{\mathbb{S}^2} \tilde{F}_0\left(n_0, \frac{v+v_*}{2}, |v-v_*|\right) G\left(\frac{v-v_*}{|v-v_*|}, t; |v-v_*|, n_0\right) \mathrm{d}n_0,
\]
where $\tilde{F}_0$ is defined by
\[
\tilde{F}_0\left(\frac{v-v_*}{|v-v_*|}, \frac{v+v_*}{2}, |v-v_*|\right)=F_0(v, v_*),\quad v\neq v_*
\]
and $\tilde{F}_0(n, v, 0)=F_0(v, v)$.

Then, a particle method for the Boltzmann equation similar to
\cite{du2025structure} is given in Algorithm~\ref{Boltzmannsphere}.
\begin{algorithm}%[!t]
	\caption{The general particle method based on Green's function}
	\label{Boltzmannsphere}
	\begin{algorithmic}[1]
		\REQUIRE Particle number $N$, initial velocity of particles $v_i^{0}$ $(i=1,\cdots, N)$, time step $ \Delta t$, terminal time $T$.
		%\STATE Set $n=0$.
		\FOR{$n=1: \lceil T/\Delta t \rceil$}
		\STATE Randomly divide $N$ particles into $N/2$ pairs.
		\STATE For each pair $i\in \{1,\cdots, N/2\}$, compute
		\begin{gather}
		q = v^{n}_{i} - v^{n}_{\theta(i)}, \quad s = v^{n}_{i} + v^{n}_{\theta(i)}, \quad n = q/{|q|}.
		\end{gather}
		\STATE  Draw a sample $n'$ from the Green's function %$G(\cdot, \Delta t; s, |q|, n)$. 
        $G(\cdot, \Delta t; |q|, n)$.
		\STATE Compute the post-collision velocity difference $q'=|q|n'$, and update the velocities 
		\begin{gather}
		v_i^{n+1} = \frac{s+q'}{2}, \quad v_{\theta(i)}^{n+1} = \frac{s-q'}{2}.
		\end{gather}
		%\STATE Increment \(n \gets n + 1\).
		\ENDFOR
		\STATE \textbf{Return} the velocities at terminal time: \(v_i^{T / \Delta t}\) (\(i = 1, \ldots, N\)).
	\end{algorithmic}
\end{algorithm}

We present several immediate comments here. 
\begin{enumerate}[(a)]
\item 
The idea of using the semigroup of the linear operator $J$ for numerical method is not new. See, for example,  \cite{medaglia2024particle}. The Green's function in \cite{medaglia2024particle}, however, is used to derive a kernel for the Landau collision and the resulted method is a DSMC type algorithm. The comparison between the proposed method and DSMC will be discussed in detail later in section~\ref{subsec:comparedsmc}.

\item 
For rotational symmetric case, the Green's function only depends on the angle $\theta$ between $n'$ and $n$ and the sampling process can then be naturally performed on this angle.
\end{enumerate}

% \subsection{Setting of the semigroup}

% should we discuss the reasonableness of the semigroup approximation?

\subsection{Spectral representation of the Green's function}\label{subsec:green}

We now gather the spectral representation of the angular semigroup and its Green's function.
 For fixed $(\bar v, |q|)$, the operator $J_{|q|}$ is rotationally invariant on $\mathbb S^2$. Consequently, spherical harmonics diagonalize its angular action, as follows from the classical Funk-Hecke formula \cite[Theorem 3]{seeley1966spherical}.
 
 % Since the present setting allows non-cutoff kernels satisfying only the finite angular momentum-transfer condition, we briefly recover the eigenvalues by angular truncation and the Funk-Hecke formula\cite[Theorem 3]{seeley1966spherical}, and then derive the corresponding Green kernel.

 % \tcr{What is the finite momentum transfer condition?}

We assume the finite momentum-transfer condition \cite{he2021boltzmann} holds, which tells that $B(1-\mu) \in L^1([-1,1])$ for fixed $|q|$. Moreover, we assume that $F_0 \in L^2(\mathbb S^2)$ for fixed $(\bar v, |q|)$ to allow the spherical harmonic expansion for $F_0$. 

One may only analysis the case $B\in L^1([-1,1])$ (for $B \notin L^1$, one could consider the truncated kernel $B_\varepsilon(\mu;|q|) = B(\mu;|q|)\boldsymbol{1}_{\{\mu\leq 1-\varepsilon\}} $, and the auxiliary function $h_\ell(\mu) = (1-P_\ell (\mu))/(1-\mu)\in C([-1,1])$ with monotone convergence theorem to derive similar results). The spherical harmonic expansion for $F_0$ gives 
\begin{equation*}
    F_0(n;\bar v,|q|)=\sum_{\ell=0}^{\infty}\sum_{m=-\ell}^{\ell}F_{\ell m}(0;\bar v,|q|)Y_{\ell m}(n), \quad \text{with } F_{\ell m}(0;\bar v,|q|) = \int_{\mathbb S^2} F_0(n;\bar v,|q|) \overline{Y_{\ell m}(n)}\mathrm dn,
\end{equation*}
where $Y_{\ell m}(n)$ is the spherical harmonic function and  $\overline{Y_{\ell m}(n)}$ is the complex conjugate of $Y_{\ell m}(n)$.

Then the Funk-Hecke formula \cite[Theorem 3]{seeley1966spherical} implies that
\begin{equation*}
    \Omega_J Y_{\ell m} = -\lambda_\ell Y_{\ell m}, \quad \text{with }  \lambda_\ell(|q|) = 2\pi \int_{-1}^1 B(\mu;|q|)(1-P_\ell(\mu))\mathrm{d}\mu \geq 0.
\end{equation*}
One with addition theorem \cite{jackson2012classical} has \begin{equation*}
    F(t, n; \bar v,|q|)=  \int_{\mathbb{S}^2} G(t,\mu;|q|)F_0(n'; \bar v,|q|)\mathrm{d}n',  \quad \text{with }  G(t,\mu;|q|) = \sum_{\ell=0}^\infty \frac{2\ell+1}{4\pi}\exp(-\lambda_\ell(|q|) t)P_\ell(\mu).
\end{equation*}
Here $G$ is the desired Green's function, which should be understood in the distributional sense whenever the kernel contains a singular part.

Denote $K(|q|):=\lim\limits_{\ell\to\infty}\lambda_{\ell}(|q|)$. Since $\lim\limits_{\ell \to \infty} P_\ell(\mu) \to 0$ for $\mu \in (-1,1)$, then for $B(\mu;|q|)\in L^1([-1,1])$, it holds that 
\[
K(|q|) = |q| \sigma_{\text{total}}(|q|).
\]
Moreover, for $B(\mu;|q|) \notin  L^1([-1,1])$, $\sigma_{\text{total}}(|q|) = \infty$ and it still holds that $K(|q|) = |q| \sigma_{\text{total}}(|q|) = \infty $. 

\begin{remark}
As we consider learning the eigenvalues from the 
scattering data. The range of $|q|$ considered in this work is finite and thus we restrict ourselves to the finite-total-cross-section regime.
\end{remark}

Note that $G(0,\mu;|q|)=\sum\limits_{\ell=0}^{\infty}\frac{2l+1}{4\pi}P_{\ell}(\mu)=\frac{1}{2\pi}\delta(1-\mu)$, hence the Green's function \(G(t,\mu;|q|)\) can be decomposed into two parts:

\begin{equation}
\begin{aligned}
   G( t,\mu;|q|)&=\frac1{2\pi}\exp(-K(|q|) t)\delta(1-\mu)+\sum\limits_{\ell=0}^{+\infty}\frac{2\ell+1}{4\pi}\Big(\exp(-\lambda_\ell(|q|) t)-\exp(-K(|q|) t)\Big)P_\ell(\mu)\\
    &:=G^{\text{d}}+G^{\text{c}}.
    \end{aligned}
    \label{eq:decomposGreen}
\end{equation}

\section{Method details}\label{sec:Algorithm}

As pointed in \cite{medaglia2024particle}, sampling from a general kernel like the Green's function is time-consuming. If the kernel $B$ is explicitly known, the proposed approach may not be efficient compared to DSMC. However, the framework could be beneficial if we do not know $B$ or the scattering kernel is an unbounded operator. In this section, we propose to infer a finite-dimensional spectral representation of the Green’s function from scattering data directly to avoid the selection of the kernel and propose a normalizing flow amortized sampling approach to reduce the complexity. See section~\ref{subsec:comparedsmc} for more discussions.  

\subsection{Inverse reconstruction of Green's function from scattering data}\label{subsec:optimization}

In this subsection we present our method to reconstruct Green's function $G( t,\mu;|q|)$. 

With~\eqref{eq:decomposGreen}, we only need to determine $\lambda_{\ell}(|q|), \sigma_{\text{total}}(|q|)$ when given $|q|, \ell$ through scattering data. Initially we clarify what we have.

For rotational symmetric case, we only care about scattering angle $\theta$, and in physics experiments the events of two particle collided with scattering angle $\theta$ given relative velocity $q$ can be detected  \cite{von2015analysis}.  Moreover, total cross sections $\sigma_{\text{total}}$ can be measured independently at different relative velocities $q$ \cite{rothe1959total}. Hence the input is the tuple $(|q_i|, \theta_i)_{i=1}^{N_1}$ and $(|q_j|,\sigma_{\text{total}}(|q_j|))_{j=1}^{N_2}$.

From the definition of differential cross section $\sigma$: \cite{bird1994molecular} \begin{equation}
\label{eq:sigma}
    \sigma(|q|,\theta)=\frac{b}{\sin\theta}\left|\frac{\mathrm{d}b}{\mathrm{d}\theta}\right|.
\end{equation}Therefore, instead of reconstructing the differential cross section $\sigma$ directly, it is sufficient to learn the deflection relation between the impact parameter $b$ and the scattering angle $\theta$, with its dependence on the relative speed $|q|$ and be smooth enough.

We decompose the impact parameter into a scale part and a normalized angular shape:
\begin{equation}
    \label{eq:decompositb}
    b(|q|,\theta)=b_0(|q|)h(|q|,\theta),\qquad b_0(|q|)\geq0,\quad h(|q|,0)=1,\quad h(|q|,\pi)=0,
\end{equation}where $b_0$ represents the maximal impact parameter and $h$ describes the normalized angular dependence. $h(|q|,\pi)=0$ is obtained by $b(|q|,\pi)=0$, corresponding to the fact that a collision with $b=0$ will cause a direct rebound. Then we will approximate $b_0, h$ respectively.

Since our target is to learn a smooth $b(|q|,\theta)$, we assume that for each fixed $|q|$, ($b(|q|,\theta)$ is monotonically decreasing in $\theta$ to ensure the smooth. Then substituting~\eqref{eq:sigma} and~\eqref{eq:decompositb} into~\eqref{eq:sigmatotal} obtains:

\begin{equation}
\label{eq:sigmatotalsimpl}
    \sigma_{\text{total}}(|q|)=\int_{\mathbb S^2}\sigma(|q|,\theta)\mathrm{d}n =2\pi\int_0^\pi\frac{b(|q|,\theta)}{\sin\theta}\left|\frac{\partial b}{\partial \theta}(|q|,\theta)\right|\sin\theta\mathrm{d}\theta =\pi b_0(|q|)^2 .
\end{equation}

Therefore, the scale function $b_0(|q|)$ is thoroughly determined by $(|q_j|,\sigma_{\text{total}}(|q_j|))_{j=1}^{N_2}$ and for each sample $(|q_j|,\sigma_{\mathrm{total}}(|q_j|))$ we define:
\begin{equation*}
    R_j=\sqrt{\frac{\sigma_{\mathrm{total}}(|q_j|)}{\pi}}.
\end{equation*}as the observation of $b_0(|q|)$.

Then we choose a positive B-spline model
\begin{equation}
\label{eq:approximatb0}
    b_{0,a}(|q|)=\exp\left(\sum_{m=0}^{M_0-1}a_m\xi_m(u(|q|))\right),\qquad u(|q|)=\frac{\log |q|-\log |q|_{\min}}{\log |q|_{\max}-\log |q|_{\min}},
\end{equation}
to approximate $b_0(|q|)$, where $\xi_m$ is B-spline, $|q|_{\min}, |q|_{\max}$ are the lower and upper cutoffs of $|q|$ in $[0,+\infty)$. The coefficient $a$ is fitted by least-squares loss on the log scale:
\begin{equation}
\label{eq:Lscale}
    \mathcal{L}_{\text{scale}}(a)=\frac1{N_2}\sum_{j=1}^{N_2}\left(\log b_{0,a}(|q_j|)-\log R_j\right)^2 .
\end{equation}

Next, we approximate the normalized shape $h(|q|,\theta)$. Let $x=\theta/\pi$ and choose basis functions
\begin{equation*}
    \chi_l(x)=x(1-x)P_{l-1}(2x-1),\qquad l=1,\dots,L.
\end{equation*} so $\chi_l(0)=\chi_l(1)=0,\quad l=1,\dots,L.$

We choose the following model to approximate $h(|q|,\theta) $, assuming a $q$-independent angular shape plus a weak $|q|$-dependence on $h_d(\theta)$ to capture the fluctuation with respect to $|q|$:
\begin{equation}
\label{eq:approximath}
    h_D(|q|,\theta)=h_d(\theta)+\sum_{m=0}^{M_h-1}\sum_{r=1}^{L}D_{mr}\hat{\xi}_m(u(|q|))\chi_r(x):=\hat{h}_D(u,\theta), \quad h_d(\theta)=1-x+\sum_{l=1}^{L}d_l\chi_l(x).
\end{equation}
where $\hat{\xi}_m$ are centered B-spline.

\begin{remark}
    If high accuracy is not required, we can just approximate the $h(|q|,\theta)$ by a $|q|$-independent function \(h_d(\theta)\). In practice, we first optimize \(h_d(\theta)\) and then introduce the $|q|$-dependent correction \(D\) based on this optimized baseline \(h_d(\theta)\).
\end{remark}

With $\frac{\partial b}{\partial\theta}=b_0(|q|)\frac{\partial h}{\partial\theta}$, we deduce $\rho_\theta(\theta\mid |q|)=-2h(|q|,\theta)\frac{\partial h}{\partial\theta}(|q|,\theta)$. Therefore, the empirical KL divergence for the angular samples $(|q_i|,\theta_i)_{i=1}^{N_1}$ is
\begin{equation}
\label{eq:LKL}
    \mathcal{L}_{\mathrm{KL}}(D)=-\frac1{N_1}\sum_{i=1}^{N_1}\log\left(-h_D(|q_i|,\theta_i)\frac{\partial h_D}{\partial\theta}(|q_i|,\theta_i)\right),
\end{equation}

To reduce oscillations in the reconstructed impact parameter $b$, we add a weak smoothness penalty in $|q|$. Specifically we penalize $u(|q|)$ rather than with respect to \(|q|\) directly. We define:
\begin{equation}
\label{eq:Lqcontinuous}
    \mathcal{L}_{|q|}^{\text{c}}(D)=\int_0^1\int_0^\pi\left|\frac{\partial \hat{h}_D}{\partial u}(u,\theta)\right|^2\mathrm d\theta\,\mathrm du .
\end{equation}

In practical implementation,~\eqref{eq:Lqcontinuous} is approximated by tensor-product Gauss-Legendre quadrature. Let \(\{u_a,w_a^u\}_{a=1}^{N_u}\) be quadrature nodes and weights on \([0,1]\), and \(\{\theta_b,w_b^\theta\}_{b=1}^{N_\theta}\) be quadrature nodes and weights on \([0,\pi]\). We use the discrete penalty
\begin{equation}
\label{eq:Lq} 
    \mathcal{L}_{|q|}(D)=\sum_{a=1}^{N_u}\sum_{b=1}^{N_\theta}w_a^u w_b^\theta\left|\frac{\partial h_D}{\partial u}(u_a,\theta_b)\right|^2 ,
\end{equation}to prevent producing non-physical oscillations in $|q|$ direction.

Eventually the total loss is:
\begin{equation}
\label{eq:LTOTAL}
    \mathcal{L}(a,D)=\mathcal{L}_{\text{scale}}(a)+\mathcal{L}_{\mathrm{KL}}(D)+ \alpha_{|q|} \mathcal{L}_{|q|}(D).
\end{equation}where $\alpha_{|q|}$ is a hyperparameter.

After $b_{0,a}$ and $h_D$ are obtained, 
the eigenvalues $\lambda_{\ell}(|q|)$ are computed directly:
\begin{equation}
\label{eq:eigenvalues}
    \lambda_{\ell}(|q|)
    =
    2\pi |q| b_{0,a}(|q|)^2
    \int_0^\pi
    h_D(|q|,\theta)
    \left|
    \frac{\partial h_D}{\partial\theta}(|q|,\theta)
    \right|
    \left(
    1-P_{\ell}(\cos\theta)
    \right)
    \mathrm{d}\theta,
\end{equation}while $\sigma_{\mathrm{total}}(|q|)$ are calculated by~\eqref{eq:sigmatotalsimpl}. See Algorithm~\ref{algorithm:distributrans} for the sorrow procedure.

\begin{algorithm}
\floatname{algorithm}{Algorithm}
\caption{Green's function reconstruction from scattering data}
\label{algorithm:distributrans}
\begin{algorithmic}[1]
\REQUIRE Angular samples $\{(|q_i|,\theta_i)\}_{i=1}^{N_1}$, total cross section samples $\{(|q_j|,\sigma_{\mathrm{total}}(|q_j|))\}_{j=1}^{N_2}$ and target relative velocity $\{|q_k|\}^{N_3}_{k=1}$.

\STATE Construct the positive scale model $b_{0,a}(|q|)$ and the normalized shape model $h_D(|q|,\theta)$ by~\eqref{eq:approximatb0} and~\eqref{eq:approximath} respectively.

\STATE Determine the parameter \(a,D\) by minimizing~\eqref{eq:LTOTAL}

\STATE For $\{|q_k|\}^{N_3}_{k=1}$, compute the eigenvalues and the total cross section by~\eqref{eq:sigmatotalsimpl} and~\eqref{eq:eigenvalues}.

\RETURN
\(
    \{\lambda_{\ell}(|q_k|)\}_{\ell=0}^{L_{\max}},\sigma_{\text{total}}(|q_k|).
\)
\end{algorithmic}
\end{algorithm}

\subsection{Conditional Monotone Flow for Sampling}
\label{subsec:NF}

From~\eqref{eq:decomposGreen}, the Green's function is decomposed into the non-collisional part \(G^{\mathrm d}\) and the collisional part \(G^{\mathrm c}\). The former corresponds to the event of no collision and can be sampled directly. Therefore, it remains to construct a sampler for \(G^{\mathrm c}\).

Since \(G^{\mathrm c}\) admits a spectral expansion, in practice we truncate it at \(L_{\max}\):
\begin{equation*}
    G^c_{L_{\max}}(t, \mu; |q|) =\sum\limits_{\ell=0}^{L_{\max}}\frac{2\ell+1}{4\pi}\Big(\exp(-\lambda_\ell(|q|) t)-\exp(-K(|q|)t)\Big)P_\ell(\mu).
\end{equation*}
The convergence and positivity of $G^c_{L_{\max}}$ are established in section~\ref{sec:truncation}. Hence, we just need to construct a sampler for $G^c_{L_{\max}}$.

Here we choose Normalizing Flow (NF) as the sampler. NF represents a complex distribution as the push-forward of a simple reference distribution through an invertible transformation \cite{RezendeM15,George2021}. Let \(z\sim p_0\) and \(y=T_\Theta(z)\). The induced density is given by the change-of-variables formula:
\begin{equation*}
    \log p_\Theta(T_\Theta(z))=\log p_0(z)-\log\left|\det\frac{\partial T_\Theta}{\partial z}(z)\right|.
\end{equation*}

In our setting, the target density is not given by samples. For each condition $c=(|q|,\Delta t)$, $G^c_{L_{\max}}$ is already available by subsection~\ref{subsec:optimization}. Hence we train a conditional flow $y=T_\Theta(z;c)$ by minimizing the KL divergence from the generated density to the known target density. Since the Green's function is defined on \(x\in(-1,1)\), we use the unbounded variable $x=\tanh(y/2)$. Therefore the target density in the \(y\) variable is
\begin{equation}
\label{eq:ytargetdensity}
    \pi_c(y)=G_{L_{\max}}^{\mathrm c}\left(\Delta t,\tanh(y/2); |q|\right)\frac12\left(1-\tanh^2(y/2)\right).
\end{equation}
Up to a constant independent of \(\Theta\), the training loss is
\begin{equation*}
    \mathcal L(\Theta)=\mathbb E_{c,z}\left[-\log \pi_c(T_\Theta(z;c))-\log\left|\frac{\partial T_\Theta}{\partial z}(z;c)\right|\right],\qquad z\sim p_0 .
\end{equation*}

The conditional flow follows the architecture shown in Figure~\ref{fig:networkconstruction}. Given the condition \(c=(|q|,\Delta t)\), a condition encoder produces an embedding \(e(c)\) then used by FiLM-conditioned parameter networks \cite{perez2018film} to generate the parameters of flow layer. The flow then consists of a conditional Sinh-Arcsinh-type transformation, \(L\) strictly monotone tanh residual layers, another conditional Sinh-Arcsinh-type transformation, and a conditional affine head. The monotone layers are designed so that their derivatives are strictly positive, which guarantees invertibility and gives an explicit log-Jacobian. The conditional Sinh-Arcsinh type transformations is motivated by thought in \cite{jones2009sinh} to improve the representation of skewness and tail behavior. See Appendix~\ref{app:nf} for the detailed of the network architecture.

After training, $G_{L_{\max}}^{\mathrm c}$ are sampled by drawing \(z^{(i)}\sim p_0\) and computing:
\[
    x^{(i)}=\tanh(T_\Theta(z^{(i)};c)/2).
\]

\begin{remark}
For NF training, \(G^c_{L_{\max}}\) is converted into a normalized density in \(\mu=\cos\theta\). The factor \(2\pi\) is included to ensure that the resulting density integrates to one on \([-1,1]\). For notational simplicity, this normalization factor is omitted.
\end{remark}

\begin{figure}[!t]
    \centering
    \includegraphics[height=65mm,width=100mm]{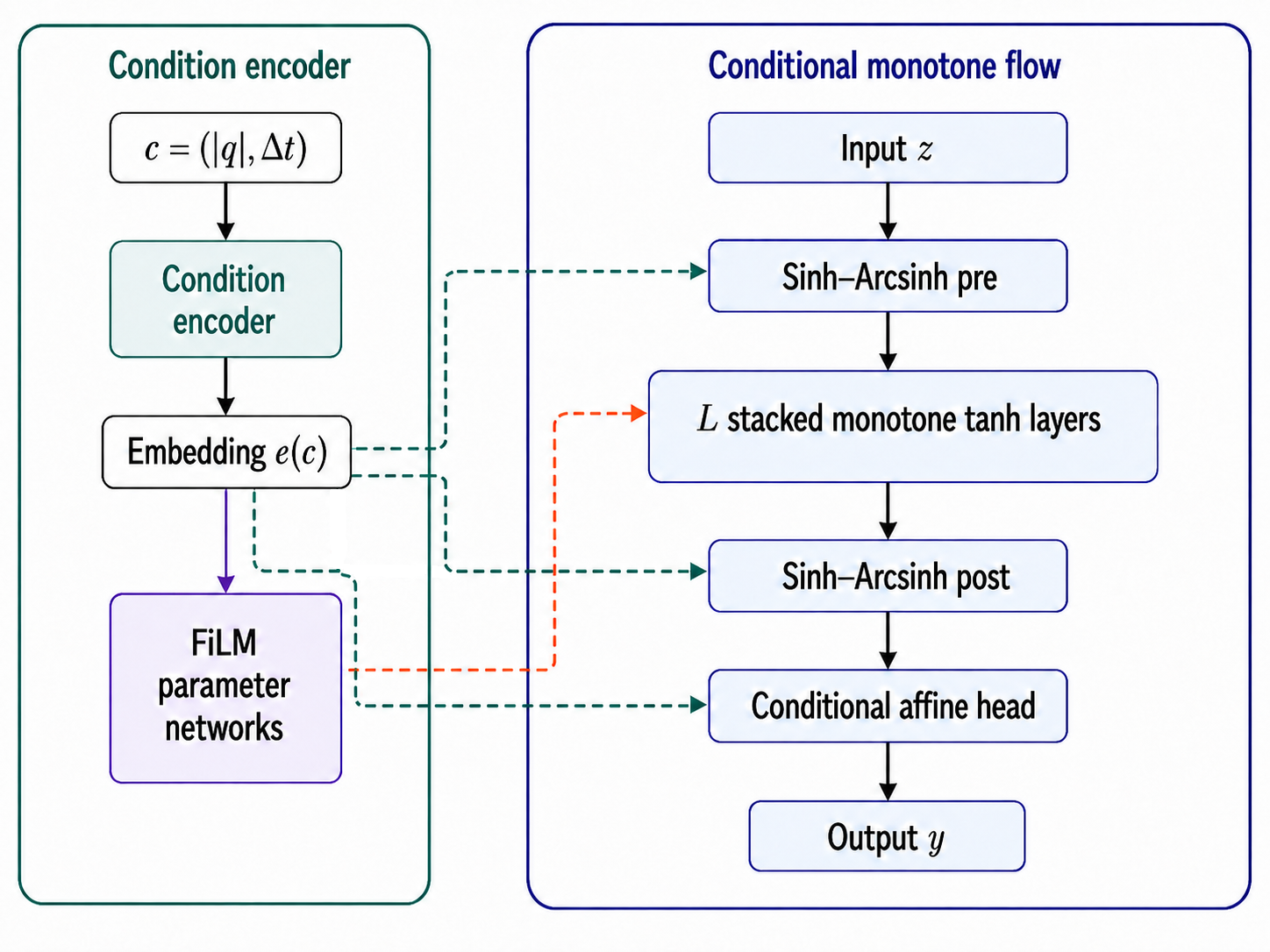}
    \label{fig:networkconstruction}
    \caption{Network Construction}
\end{figure}

\subsection{Boltzmann solver}\label{subsec:solveBolt}

In this subsection we illustrate how to assemble subsection~\ref{subsec:optimization} and~\ref{subsec:NF} into the fourth step of Algorithm~\ref{Boltzmannsphere} to solve spatially homogeneous Boltzmann equation and applied to spatially non-homogeneous Boltzmann equation.

With sampler of $G^c_{L_{\max}}$ constructed in subsection~\ref{subsec:NF} and non-collision case, we can sample the scattering angle $\theta$ under the proportion of two cases and from $\text{UNIF}[0,2\pi)$ sample azimuthal angle $\varphi$ due to the rotational symmetry. Subsequently we can use $n,\theta,\varphi$ to reconstruct $n'$. Here the proportion of non-collision and collision case is $\exp\left(-K(|q|)\Delta t\right):1-\exp\left(-K(|q|)\Delta t\right)$ from~\eqref{eq:decomposGreen}. See Algorithm~\ref{algorithm:Greensample} for the detailed procedure of the fourth step of Algorithm~\ref{Boltzmannsphere}.

\begin{algorithm}
\floatname{algorithm}{Algorithm}
\caption{Green's function sampling}
\label{algorithm:Greensample}
\begin{algorithmic}[1]
\REQUIRE Time step $\Delta t$, relative velocity $q$, pre-collision direction $n$, trained NF model $T_{\Theta}$, optimized $\sigma_{\text{total}}$.
\STATE Sample $p$ from $\text{UNIF}[0,1)$ 

\IF{$p<\exp\left(-K(|q|)\Delta t\right)$}
    \RETURN $n$.
\ENDIF

\STATE Sample $\theta$ by NF with $T_{\Theta}$ and condition $c=(|q|,\Delta t)$; sample $\varphi$ from $\text{UNIF}[0,2\pi)$.

\STATE Choose two unit vectors \(e^{(1)},e^{(2)}\) such that\(\{n,e^{(1)},e^{(2)}\}\)
forms an orthonormal basis. Calculate the post-collision direction
\[
    n'=\cos\theta n+\sin\theta\left(\cos\varphi e^{(1)}+\sin\varphi e^{(2)}\right).
\]

\RETURN $n'$.
\end{algorithmic}
\end{algorithm}

For spatially non-homogeneous cases, we adopt a time-splitting approach that decouples the Boltzmann equation to collision step and the advection step:
\begin{subequations}
\begin{align}
    \partial_t f = Q(f, f), \label{eq:collpart} \\
    \partial_t f + v \cdot \nabla_x f = 0. \label{eq:advectpart}
\end{align}
\end{subequations}Then the well-known PIC algorithm is employed \cite{bailo2024collisional, serikov1999particle}. As a preparatory step for the algorithm, the spatial domain \(\Omega\) is divided into a set of grids \(\Omega_k\). For~\eqref{eq:collpart}, we use Algorithm~\ref{Boltzmannsphere} to update particle velocities in each grid, preserving the kinetic energy \(\mathcal{E}_K\). For~\eqref{eq:advectpart}, it is just pure transport with constant velocity and particles positions are updated easily.

\section{Estimates of the truncated Green's function}\label{sec:truncation}

For particle simulation, a finite spectral approximation of the Green's function must be adopted. We study in this section the truncation of the Green's function (used in the training of NF). We restrict the analysis in this section to the case with finite total cross section (in the method the reconstructed impact-parameter model has a finite maximal scale $b_0(|q|)$ so that the total cross section is finite). 

% Although $B_{\text{eff}}(|q|, \cos \theta)$ needs not be reconstructed explicitly in the numerical procedure, it is useful to introduce it in the analysis as the collision kernel induced by the learned impact-parameter relation. We define
% \begin{equation*}
%     B_{\text{eff}} (|q|, \cos \theta)=|q| \frac{b_{0,a}(|q|)^2h_D(|q|, \theta)}{\sin \theta} \bigg | \frac{\partial h_D(|q|, \theta)}{\partial \theta}\bigg|.
% \end{equation*}
% Then
% \begin{equation*}
%     K_{\text{eff}} = 2\pi\int_{[-1,1]}B_{\text{eff}}(|q|,\mu)\mathrm{d}\mu=|q|\sigma_{\text{total}}(|q|)<\infty,
% \end{equation*}
% and the eigenvalues computed by Algorithm~\ref{algorithm:distributrans} satisfy
% \begin{equation*}
%     \lambda_\ell(|q|)=2\pi\int_{[-1,1]}B(|q|,\eta)(1-P_\ell(\mu))\mathrm{d}\mu.
% \end{equation*}

% For notational simplicity, we write write $B$ and $K$ for $B_{\text{eff}}$ and $K_{\text{eff}}$ throughout this section.

Recalling the decomposition for Green's function $G$ in \eqref{eq:decomposGreen}, the atomic component is nonnegative and it suffices to sample from the continuous component. It remains to study the truncated kernel
\begin{equation*}
    G^c_N( t, |q|,\mu) =\sum\limits_{\ell=0}^{N}\frac{2\ell+1}{4\pi}\Big(\exp(-\lambda_\ell(|q|) t)-\exp(-K(|q|)t)\Big)P_\ell(\mu).
\end{equation*}
Here $\lambda_\ell\ge 0$ is the opposite of the eigenvalue and $K \ge 0$ is the $\ell \to \infty$ limit of $\lambda_\ell$ (see subsection \ref{subsec:green}).

\begin{remark}
    In this subsection, \(N\) represents the truncation order of the expansion of \(G^c\), rather than the number of particles. In other subsections we use \(L_{\max}\) to denote this quantity.
\end{remark}
% and $B$ is the collision kernel related to the reconstructed impact-parameter model.

For convenience, we suppress the dependence on $|q|$ in the notation. The results below will be uniform for a compact interval $[q_1, q_2]\subset (0,\infty)$.
The estimates could blow up for $q$ close to zero, which however, is not important for practical simulations. 

In the following, we assume that $B\ge 0$, $B \not\equiv 0$ and $B \in L^2([-1,1])$.
Under the assumption above, one could consider the Legendre expansion for $B$ as $B(\mu) = \sum\limits_{\ell = 0}^\infty B_\ell P_\ell(\mu)$ in $L^2$-sense. Then
$K=4\pi B_0$ and denote 
\[
\Lambda_\ell = K- \lambda_\ell = \frac{4\pi}{2\ell+1} B_\ell.
\]
Since $B\in L^2([-1, 1])$, one then has $\sum\limits_{\ell}\frac{2}{2\ell+1}B_{\ell}^2<\infty$. With the fact
\begin{equation}\label{eq:exp_estimate}
    \left|\exp(-\lambda_\ell(|q|) t)-\exp(-K(|q|)t)\right|\le |\lambda_\ell-K| t=|\Lambda_\ell| t,
\end{equation}
it is straightforward to see that $G^c(t, |q|, \cdot)\in C([0, T]; L^2([-1, 1]))$ and thus it can be normalized to a probability density on $\mathbb{S}^2$. Moreover, since
\[
\int_{-1}^1 G^c \,\mathrm{d}\mu=\frac{1}{2\pi}(1-e^{-Kt}),
\]
Hence,
$t\mapsto p^c(t, \mu):=\frac{G^c}{\|G^c\|_{1}}$ (and thus the normalized probability density on $\mathbb{S}^2$ corresponding to $G^c$)
is also regularly behaved at $t=0$.

We first note the following simple fact.
\begin{lemma}\label{lemma:positivity}
The normalized probability density $p^c(t, \mu)$ is positive for $t>0$ and has a positive lower bound on any time interval $[\delta, T]$ with any $0<\delta<T$.
\end{lemma}
\begin{proof}
Denote $p(\mu) =p(n,n'):= B(\mu)/K$ the normalized single-collision density. By the compound-Poisson expansion
\begin{equation}\label{eq:Bpoisson}
    p^c(t,\mu) = \frac{2\pi \exp(-Kt)}{1-e^{-Kt}}\sum_{r=1}^\infty \frac{t^r}{r!}B^{(*r)}, \text{ with } B^{(*r)} = K^{r}p^{(*r)},
\end{equation}
where $p^{(*1)} = p$ and 
\begin{equation*}
    p^{(*(r+1))} (n,n') = \int_{\mathbb S^2} p^{(*r)} (n,\xi)p(\xi,n')d\xi
\end{equation*}
describes the probability of transferring from $n$ to $n'$ with exact $r$ angular scattering.

The Legendre moment is given by
\begin{equation*}
    a_\ell = \int_{\mathbb S^2} p(n, n') P_\ell(n\cdot n')\mathrm d n' = \frac{\Lambda_\ell}{K}, \quad \text{ with } a_0 = 1.
\end{equation*}
For fixed $\ell \geq 1$, one has $|a_\ell|\leq \int p|P_\ell| <1$ since $\{\mu: |P_\ell(\mu)| = 1\}$ is a null set. Moreover, since $p\in L^2(\mathbb S^2)$, with Cauchy-Schwarz inequality,
\begin{equation*}
    |a_\ell| \leq \|p(n, \cdot)\|_2 \|P_\ell(n,\cdot)\|_2=\|p\|_2\sqrt{\frac{4\pi}{2\ell+1}} \to 0,
\end{equation*}
which implies $\rho:=\sup_{\ell\geq 1} |a_\ell|<1$.
Spherical harmonic addition theorem \cite{jackson2012classical} along with summarization implies
\begin{equation*}
    p^{(*r)}(n,n') =\frac{1}{4\pi} + \sum_{\ell = 1}^\infty \frac{2\ell+1}{4\pi} a_\ell^r P_\ell (n\cdot n').
\end{equation*}
Since $p \in L^2$, then Parseval's identity gives
$ A:= \sum\limits_{\ell = 1}^\infty \frac{2\ell+1}{4\pi} |a_\ell|^2 < \infty$.
Hence for $r\geq 2$, there exists
\begin{equation*}
    \left\|p^{(*r)}-\frac{1}{4\pi}\right\|_\infty = \text{ess} \sup \left| p^{(*r)}(n,n')-\frac{1}{4\pi}\right|\leq \rho^{r-2}\sum_{\ell = 1}^\infty \frac{2\ell+1}{4\pi} |a_\ell|^2=A\rho^{r-2}.
\end{equation*}

Choose $r_*$ such that $A\rho^{r_*-2}< 1/(8\pi)$, then $p^{(*r_*)}(n,n')>1/(8\pi)$ for all $n$ and $n'$. Along with~\eqref{eq:Bpoisson}, one has

\begin{equation*}
    p^c(t,\mu) \geq \frac{2\pi}{1-\exp(-KT)} \exp(-KT) (K\delta)^{r_*}/(8\pi(r_* !))>0, \quad\forall \mu \text{ and } t\in[\delta,T],
\end{equation*}
which concludes the proof.
\end{proof}

Next, we first establish convergence and truncation-error estimates for $G_N^c$, and then investigate conditions under which the truncated continuous component remains nonnegative.

\begin{theorem}[Convergence theorem]
    Under the assumption above, one has the following convergence for the truncated continuous component of the Green's function  for $t\in[0,T]$
    \begin{equation*}
        \|G^c-G^c_N\|_2+
        \|p^c-p_N^c\|_1  \to 0,\quad N\to\infty,
    \end{equation*}
where $p_N^c:=G^c_N/\|G_N^c\|_1$ and $S_NB(\mu)=\sum_{\ell=0}^N B_\ell P_\ell(\mu)$.

Additionally, if the Legendre projections of $B$ converge uniformly to $B$, i.e.
 \begin{equation*}
        \|B-S_NB\|_\infty \to 0,
\end{equation*}
then there exists a  $L^\infty$-error estimation
\begin{equation}\label{eq:trunerr}
        \|G^c-G^c_N\|_\infty \leq T\|B-S_NB\|_\infty + \pi T^2\|B-S_NB \|_2^2 \to 0.
\end{equation}
Moreover, for each $\delta\in (0, T)$, there exists $N$ such that for all $n>N$, the truncation $G^c_n(t,\mu)$ and thus $p_n^c$ are positive for $t\in[\delta,T]$.
\end{theorem}

\begin{proof}  For the $L^2$-error estimation for the Green's function, since $G^c(t, |q|, \cdot)\in C([0, T]; L^2([-1, 1]))$, one can use the Parseval's identity  fixed $t$ with the fact~\eqref{eq:exp_estimate} to show
\begin{equation*}
    \|G^c-G^c_N\|_2^2\leq t^2 \sum_{\ell>N} \frac{2}{2\ell+1}\left(\frac{2\ell+1}{4\pi}\right)^2\Lambda_\ell^2 =  t^2 \sum_{\ell>N} \frac{2}{2\ell+1} B_\ell^2 = t^2 \|B-S_NB\|^2.
\end{equation*}
It follows that
\[
\|G_N^c-G^c\|_1\le \sqrt{2}\|G^c-G_N^c\|_2
=o(1)\|G^c\|_1,
\]
for $t\le T$. The first result then follows.

For the $L^\infty$-error estimation, for fixed $t$ a direct calculation gives 
\begin{equation*}
         G^c_N( t, \mu) =\exp(-Kt)\sum\limits_{\ell=0}^{N}\frac{2\ell+1}{4\pi}\Big(\exp\left(\Lambda_\ell t\right)-1\Big)P_\ell(\mu)        = t\exp(-Kt) S_NB(\mu)+R_N(t,\mu),
\end{equation*}
where
\begin{equation*}
    R_N(t,\mu) = \exp(-Kt)\sum\limits_{\ell=0}^{N}\frac{2\ell+1}{4\pi}\Big(\exp\left(\Lambda_\ell t\right)-1-\Lambda_\ell t\Big)P_\ell(\mu).
\end{equation*}
Here we split $G_N^c$ into linear term and second-order remainder term.

For~\eqref{eq:trunerr} it suffices to derive the error estimation for the second-order remainder term $R_N$.
Since $B\geq 0$ and $|P_\ell| \leq 1$, one has $ |\Lambda_\ell| \leq K$. Moreover, Taylor's expansion tells $|e^x-1-x|\leq  x^2e^{|x|} /2$, which further implies
\begin{equation*}
    \exp (-Kt)\left| \exp(\Lambda_\ell t)-1- \Lambda_\ell t\right|\leq \Lambda_\ell^2t^2/2.
\end{equation*}

With the inequality above it follows 
\begin{equation*}
    \|R\|_\infty \leq \sum_{\ell =0}^\infty \frac{2\ell+1}{4\pi}\frac{\Lambda_\ell^2t^2}{2}=\sum_{\ell =0}^\infty \frac{2\pi t^2 B_\ell^2}{2\ell+1}=  \pi t^2\|B\|_{L^2}^2.
    \end{equation*}
where the first equality uses the definition and the second equality uses Parseval's identity.
Hence by Weierstrass M-test, $R_N$ converges to $R$ absolutely and uniformly.

Similarly it also holds
\begin{equation*}
    \|R-R_N\|_\infty \leq \sum_{\ell > N} \frac{2\ell+1}{4\pi}\frac{\Lambda_\ell^2t^2}{2}=\sum_{\ell > N} \frac{2\pi t^2 B_\ell^2}{2\ell+1}=  \pi t^2\|B-S_NB\|_{L^2}^2.
\end{equation*}
Taking the supremum for $t\in[0, T]$ gives the desired inequality.

It remains to prove that there exists $N$ such that for all $n>N$, the truncation $G^c_n(t,\mu)$ is positive for $t\in[\delta,T]$, which is implied by Lemma \ref{lemma:positivity} and the uniform convergence of $G_N^c$.
\end{proof}

We should note here that there are many known conditions for the convergence of the Legendre projections in $L^\infty$-norm, and one could refer to \cite{wang2021much} for more details.

\section{Discussions and comments}\label{sec:discom}

In this section, we perform several discussions to deepen the understanding of the method.

\subsection{The asymptotic chaos and formal analysis}\label{subsec:asymptot}

Note that Algorithm~\ref{Boltzmannsphere} is in the form of the random batch method in \cite{jin2020random}, and enjoys the same benefit such as linear complexity, high parallel efficiency etc. The analysis for the large particle limit in \cite{jin2022mean} can then be applied without change to our method here.  Let $F^N$ be the joint distribution of the $N$ particles in our method and $F^{(N,k)}$ be the $k$-marginal. For fixed time step $\Delta t$, when $N$ is large enough, the two-particle marginal distribution $F^{(N,2)}$ can be decomposed into the form
\[
F^{(N,2)}(t_k)=f_k^{\otimes 2}+g
\]
where $g$ vanishes as $N\to\infty$, and $f_k$ satisfies the following large particle limit
\begin{algorithm}[H]
\caption{Large particle limit of the method}
\label{algorithm:largelimit}
\begin{algorithmic}
\STATE Take $\eta>0$, $t_k= k \eta$.

\FOR{$k=0,1,2,\cdots$}
\STATE Set $F_k=f_k^{\otimes 2}$
 
\STATE Evolve the linear equation
\[
\partial_t F=J F,
\]
\STATE Set $f_{k+1}=\pi F_{k+1}^-=\int F_{k+1}^-(x, y)\,dy$
\ENDFOR
\end{algorithmic}
\end{algorithm}

Then, $F^{(N,1)}(t_k)=f_k+\tilde{g}$ where $\tilde{g}$ vanishes as $N\to\infty$. It then follows that
\[
\begin{split}
F^{(N,1)}(t_{k+1})
&=\int e^{\Delta t J}F^{(N, 2)}(t_k) dv_*
=\int [e^{\Delta t J}f_k^{\otimes 2}] dv_*+r_N
=\int [e^{\Delta t J}(F^{(N,1)}(t_k))^{\otimes 2}] dv_*+\tilde{r}_N \\
&=F^{(N,1)}(t_k)+\Delta t \int J (F^{(N,1)}(t_k))^{\otimes 2} dv_*+\tilde{r}_N+o(\Delta t)\\
&=F^{(N,1)}(t_k)+\Delta t Q(F^{(N,1)}(t_k),F^{(N,1)}(t_k))+\tilde{r}_N+o(\Delta t),
\end{split}
\]
where $r_N\to 0$ as $N\to\infty$.
This formal analysis then justifies that our particle method is expected to approximate the Boltzmann evolution. Of course, rigorous analysis involves the establishment of the propagation of chaos uniform in $\Delta t$ and the regularity estimates of the density, which need further exploration in the future.

\subsection{Comparison to DSMC}\label{subsec:comparedsmc}

In this subsection, we perform comparison of the proposed method to DSMC. At the first glance, they are very similar as they both sample the post-collision relative velocities; they both have linear complexity etc. The method in Algorithm~\ref{Boltzmannsphere} is in the form of the random batch method in \cite{jin2020random}. All particles are updated according to the Green's function. In the DSMC type methods, the particles are chosen with probability to be updated according to the Green's function.  Below, we would like to make more intrinsic comparisons, in the methodology and in the situations they may apply.
\begin{enumerate}[(a)]
\item As is clear in literature \cite{bird2013dsmc,bird1976molecular}, the DSMC simulates the one-collision directly. For this reason, one needs the kernel $B$ to be normalized to probability distribution for the angle of change. The proposed method in this paper is based on the transition probability, so it models the statistical effect of the post-collision relative velocity and this is naturally a probability distribution, without any requirement on the kernel. In many DSMC algorithms, one needs to truncate the kernel so that it is normalizable to a probability distribution. Moreover, if the scattering operator is unbounded (like with differential operator in Landau equation), one needs a lot more effort to implement the DSMC, and the proposed method is advantageous.

Put it in the language of time-continuous Markov chain, the DSMC is modeling the $Q$-matrix, or the embedded discrete time Markov chain, while the proposed method is modeling the time-continuous chain with time $\Delta t$ directly. 

\item The DSMC must know the kernel $B$ to be implemented while the proposed method here only needs modeling on the Green's function so it could possibly avoid the modeling error in choosing the kernel family for $B$. 
In real applications, only the scattering data are available (and possibly with noise) so the proposed method could be beneficial and has better robustness by modeling the Green's function.

\item Compared to DSMC, sampling from the Green's function for Boltzmann equation in the general case is time-consuming. The work in \cite{medaglia2024particle} uses some surrogate so that sampling becomes feasible but the surrogate only works for the grazing regime. We make use of the normalizing flow amortized sampling so that sampling from the Green's function becomes feasible. Even with the normalizing flow amortized sampling, if the kernel $B$ is known explicitly, the proposed approach is still slightly more time-consuming compared to the traditional sampling in DSMC as the dimension is low.  As clarified, the proposed method aims for different situations (only the scattering data are available and one hopes to avoid the modeling error in selecting $B$). 
\end{enumerate}

As have explained, the proposed method and the DSMC are used for different situations, we will not compare the proposed method with DSMC in numerical experiments below.

\section{Numerical experiments}\label{sec:num}

In this section, we validate the effectiveness of our method through numerical experiments for calculating eigenvalues from data with noise, training a NF sampler and solving the Boltzmann equation with different collision kernels and initial distributions.

To visualize our particle solution and compare it with the exact (or reference) solution, we use a mollified solution \(f^N_{\epsilon}\) from the empirical measure of particle velocities as follows:
\begin{gather}
f^N_{\epsilon}(v) := \Psi_{\epsilon} \ast \bigg( \frac{1}{N} \sum_{i=1}^N \delta_{v_i} \bigg) = \frac{1}{N} \sum_{i=1}^N \Psi_\epsilon(v - v_i),
\end{gather}
where \(\Psi_{\epsilon}(v)\) is the Gaussian mollifier, $\epsilon$ is a parameter independent of solution. We choose $\epsilon=0.01$ in the following numerical experiments.

When there does not exist analytical exact solution, we use DSMC as our reference method with small $\Delta t$ and large $N$. To quanti‌fy the accuracy of our solution, we use the relative \(L_2\) error $\varepsilon_{\text{rel}}$ computed on a uniform mesh grid with center points \(v_l^c\). 

\subsection{Optimization and NF Learning}
\label{subsec:optimiznf}

We first test Algorithm~\ref{algorithm:distributrans} by examining the accuracy of the reconstructed collision-kernel eigenvalues \(\lambda_{\ell}(|q|)\) at prescribed relative speeds \(|q|\).

As benchmark examples, we use the VSS model \cite{koura1991variable}, for which
\begin{equation}
\label{eq:vss}
    \sigma(|q|,\theta)
    =
    C\frac{\beta+2}{2}
    |q|^{\alpha-1}
    \cos^\beta\left(\frac{\theta}{2}\right),
    \qquad
    \sigma_{\mathrm{total}}(|q|)
    =
    4\pi C |q|^{\alpha-1}.
\end{equation}
We consider two test cases with $C=\frac{1}{4\pi}$ commonly, and $(\alpha,\beta)$ are set as $(0,0),(0.384,1.3248)$ respectively. The first case corresponds to an isotropic Maxwell molecule model, while the second one is the standard VSS parameter set for argon \cite{koura1991variable}.

For each test case, we generate \(100,000\) angular samples and \(100,000\) total cross section samples on the relative speed interval \([0.05,10]\) with noise added to the data. During fitting we use $L=4, M_0=10, M_h=5, \alpha_{|q|}=5$. The resulting finite-dimensional optimization problem is solved by L-BFGS-B.

\begin{figure}[!t]
        \centering
        \subfigure[$\alpha=0,\beta=0$]{
        \centering
        \includegraphics[height=130mm,width=70mm]{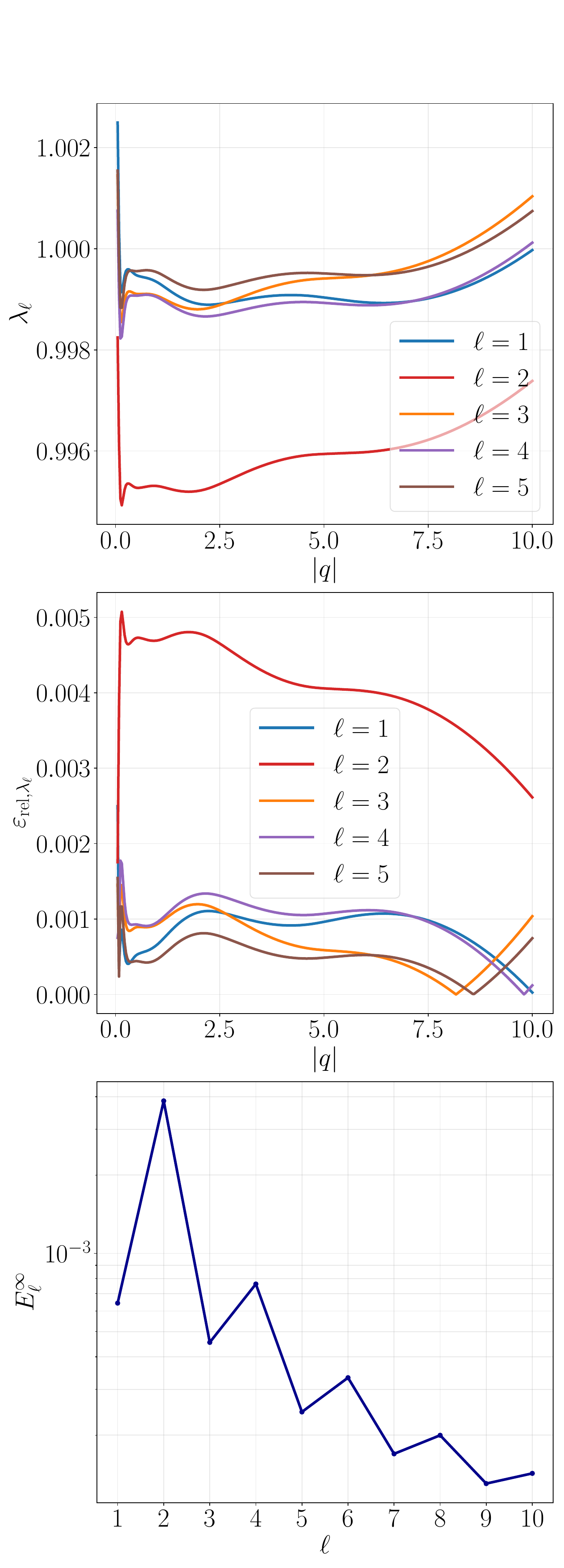}
        \label{fig:alpha0beta0optimiz}
    }
        \centering
        \subfigure[$\alpha=0.384,\beta=1.3248$]{
        \centering
        \includegraphics[height=130mm,width=70mm]{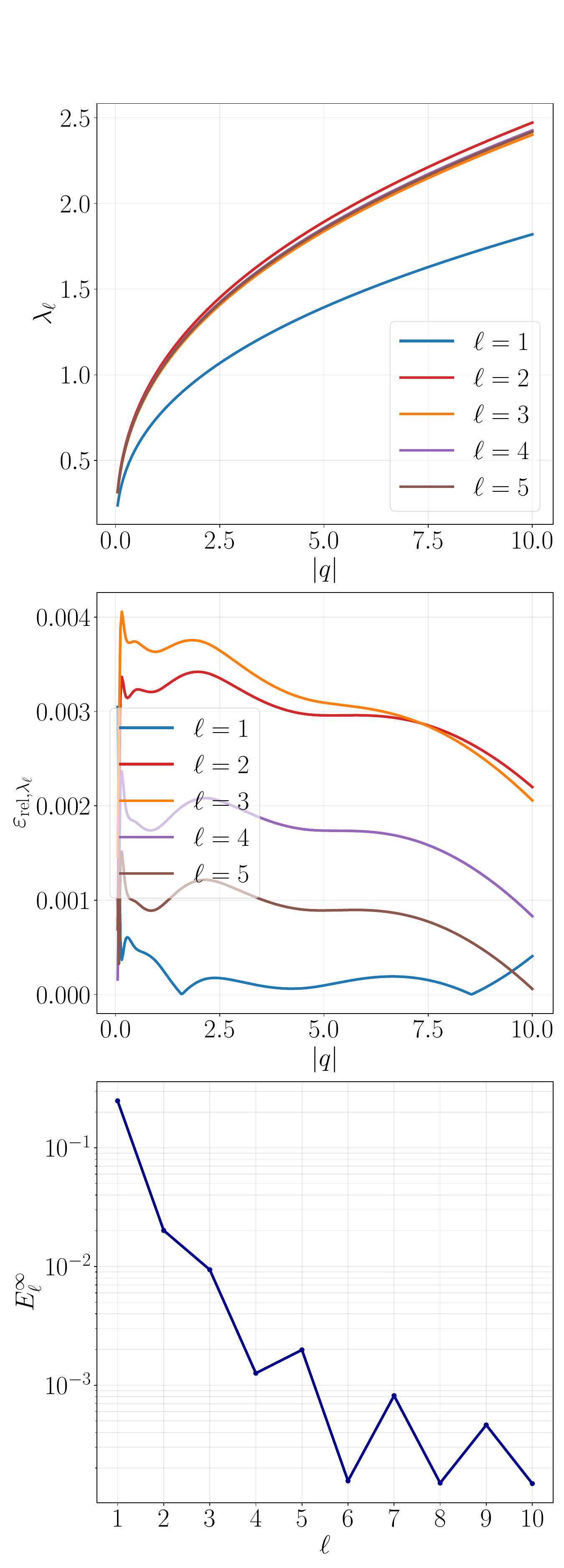}
        \label{fig:alpha0384beta13248optimiz}
    }
        \caption{Optimization results for the two VSS test cases. In each column,  The first row compares the reconstructed and exact eigenvalues \(\lambda_{\ell}(|q|)\) for selected orders \(\ell=1,\cdots,5\). The second row shows the relative error of reconstructed eigenvalues\(\lambda_{\ell}(|q|)\) for selected orders \(\ell=1,\cdots,5\). The third row shows the large-\(\ell\) consistency error \(E_{\ell}^\infty\) for \(\ell=1,\ldots,10\).}
		\label{fig:twooptimiz}
    \end{figure}

Figure~\ref{fig:twooptimiz} shows that Algorithm~\ref{algorithm:distributrans} accurately reconstructs the eigenvalues \(\lambda_{\ell}(|q|)\) for all two test cases. To check the convergence of $\lambda_{\ell}{(|q|)}$ to $K(|q|)$, we therefore define:
\begin{equation}
\label{eq:lconsisterr}
    E_l^\infty=\Bigg(\int_{0.05}^{10}\left|\lambda_l(|q|)-K(|q|)\right|\,\mathrm d|q|\Bigg)\bigg/\Bigg(\int_{0.05}^{10}\left|K(|q|)\right|\,\mathrm d|q|\Bigg).
\end{equation}
The decay of \(E_{\ell}^\infty\) further confirms that the reconstructed eigenvalues have the correct large-order behavior. Additional details on the noise setting and the reconstructed impact parameter \(b\) are reported in Appendix~\ref{app:optimizNF}.

Subsequently, we train the conditional NF sampler for the two test cases using the reconstructed eigenvalues obtained above. In the Green's function approximation, we truncate at \(L_{\max}=10\). The network is trained with AdamW and a cosine learning-rate schedule with linear warmup for \(5,000\) epochs. The initial distribution is $\mathcal{N}(0.5,2.25)$ and the conditions are sampled from $[0.05, 10], [0.001, 10]$ respectively. 

Figure~\ref{fig:twotrain} shows the sampling results after training. The NF samples agree well with the corresponding continuous part of Green's function $G^{\text{c}}_{L_{\max}}$. In particular, the discrepancy between the NF-generated samples and the analytical reference $G^{\text{c}}_N$ is small. This confirms that the trained conditional NF provides an accurate sampler for the collision Green's function. In the following Boltzmann simulation, we reuse the trained NF as the sampler for \(G^{\text{c}}_{L_{\max}}\).

\begin{figure}[!t]
        \centering
        \subfigure[Relative Entropy]{
        \centering
        \includegraphics[height=55mm,width=110mm]{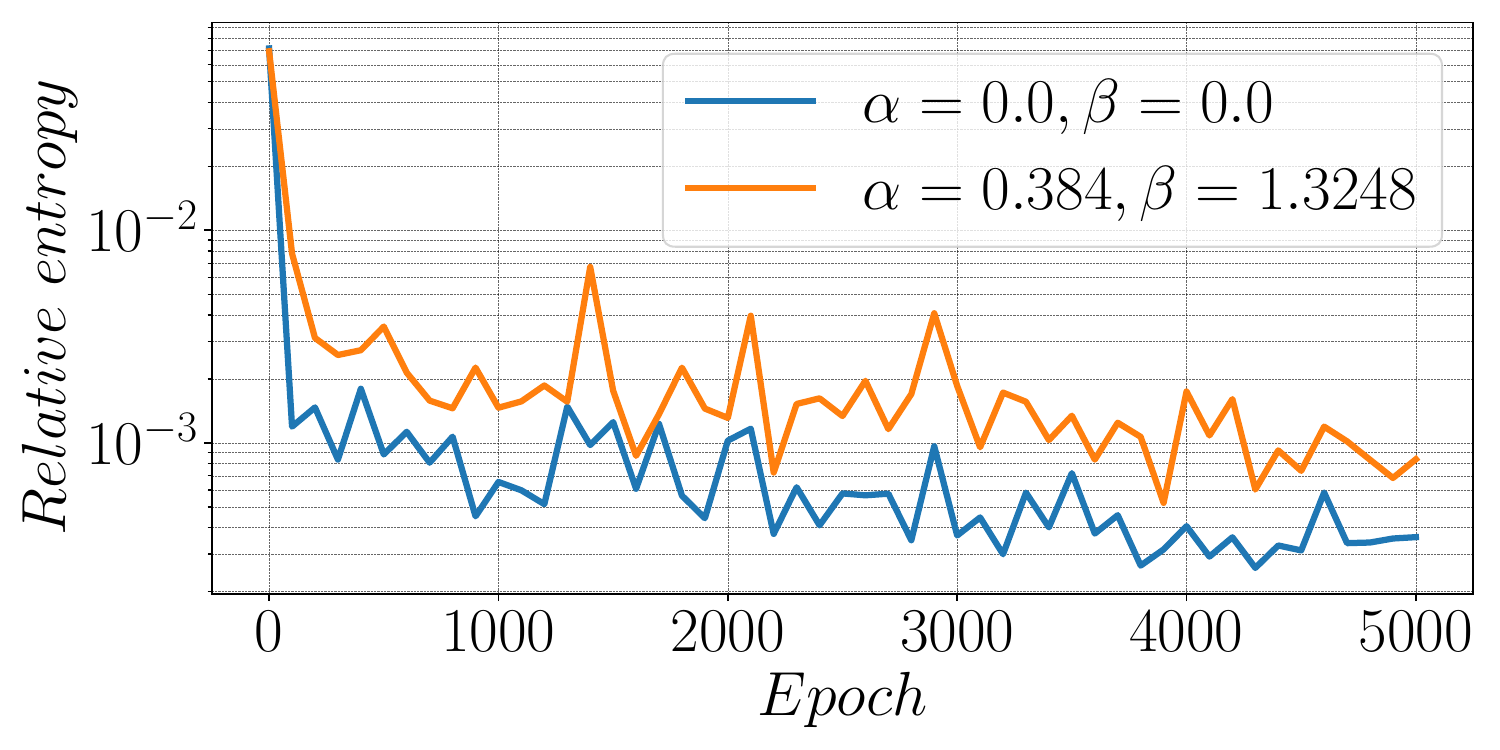}
        \label{fig:relatentromax}
    }
        \centering
        \\
        
        \subfigure[$\alpha=0,\beta=0$]{
        \centering
        \includegraphics[height=55mm,width=70mm]{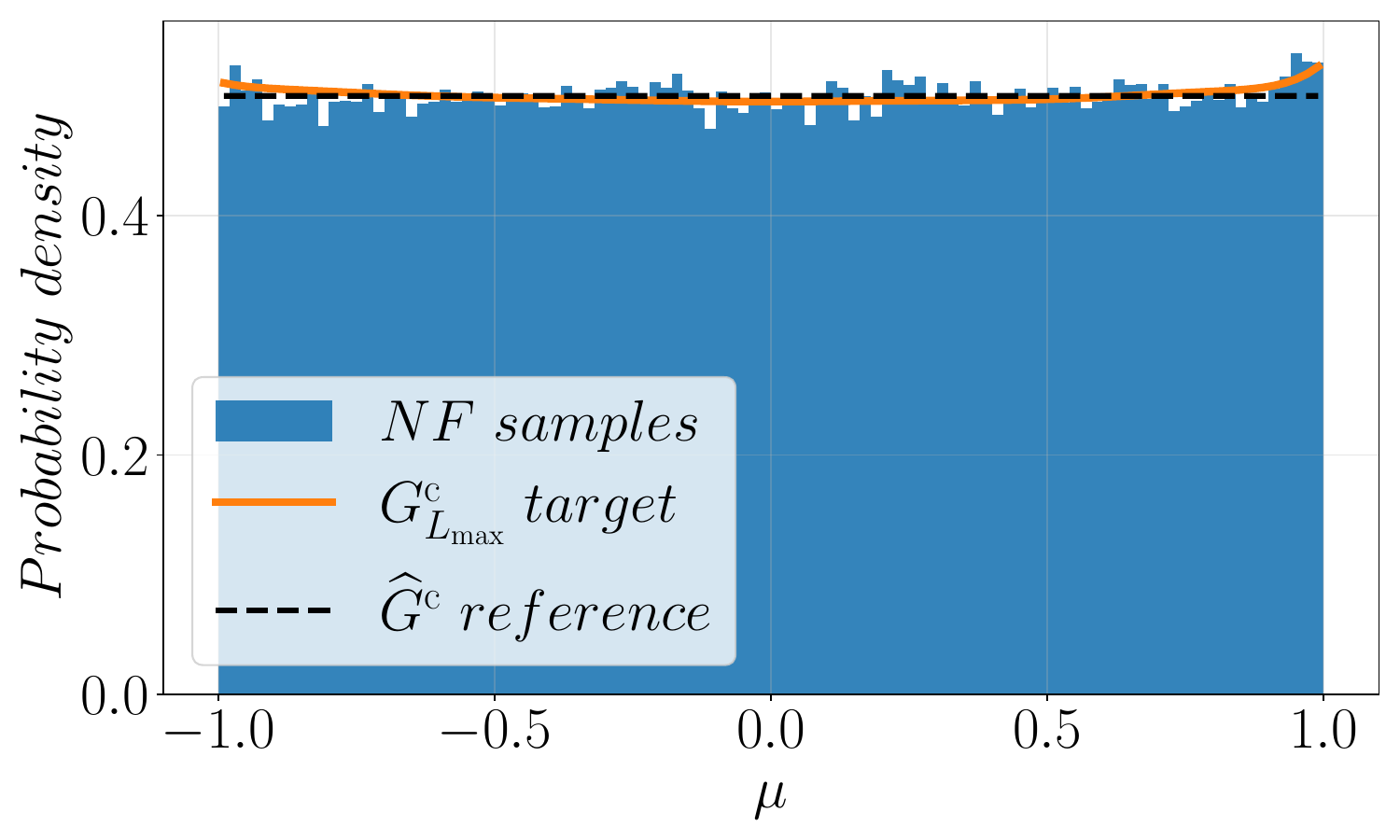}
        \label{fig:alpha0beta0forec}
    }
        \subfigure[$\alpha=0.384,\beta=1.3248$]{
        \centering
        \includegraphics[height=55mm,width=70mm]{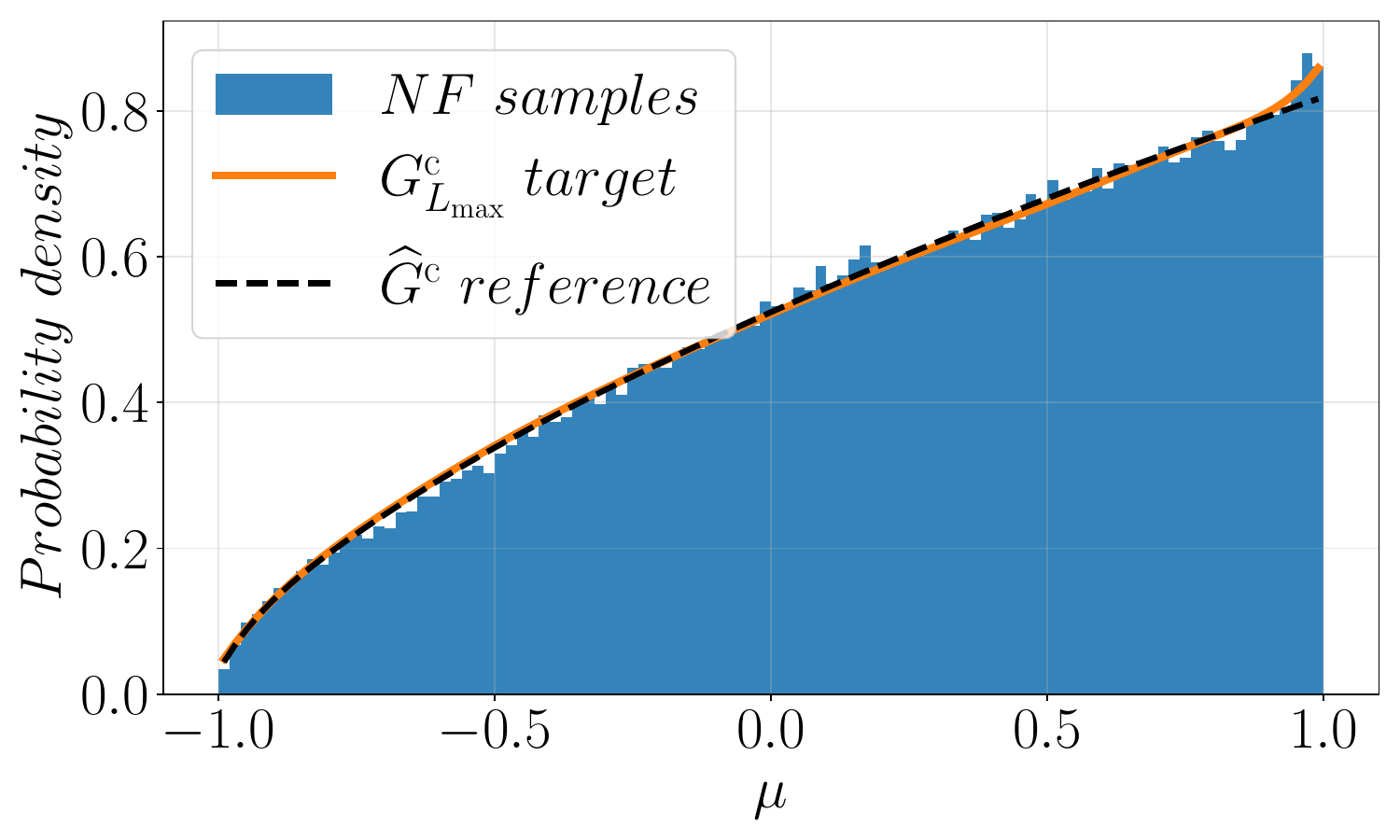}
        \label{fig:alpha0384beta13248forec}
    }
		\centering
		\caption{The first row shows the maximum of relative entropy in $q\in\{0.05,1,10\},\Delta t\in\{0.001,0.1,10\}$ with $250,000$ samples with $G^{\text{c}}_N$. The second row compares the histogram of $q=1,\Delta t=0.1$ with $250,000$ samples with $G^{\text{c}}_N$ and $\widehat{G}^{\text{c}}$ ($\widehat{G}^{\text{c}}$ refers to the continuous part of Green's function by theory) in two cases.}
		\label{fig:twotrain}
    \end{figure}

\subsection{3D BKW solution for Maxwell molecules}\label{subsec:3DBKW}

In this case, the collision kernel is $$B(|q|,\theta)=\frac{1}{4\pi}$$ and the BKW solution is given by \cite{krook1977exact}
$$ f(t,v)=\dfrac{1}{(2\pi K)^{1.5}}\Big(2.5-\dfrac{1.5}{K}+\dfrac{1-K}{2K^{2}}|v|^{2}\Big)\exp\Big(-\dfrac{|v|^{2}}{2K}\Big), \quad K=1-\exp\Big(-\frac{t-6\log(0.4)}{6}\Big) .$$ 

We set \(t_0 = 0\), \(t_\text{end} = 10\), and \(\Delta t = 0.1\). The initial particle velocities are sampled independently from the initial distribution. NF we used is the trained model with $(\alpha,\beta)=(0,0)$ in subsection~\ref{subsec:optimiznf}. We use Algorithm~\ref{Boltzmannsphere} to solve this equation with \(N = 500,000\) and \(N = 5,000,000\) (the same trained NF model is used for both cases). 
The results are presented in Figure~\ref{fig:CompardifferN3dMax} and~\ref{fig:3dMaxcross}. It confirms that our method can obtain high accuracy and conserves energy strictly.

 \begin{figure}[!t]
        \centering
        \subfigure[Relative $L_2$ error]{
        \centering
        \includegraphics[height=60mm,width=75mm]{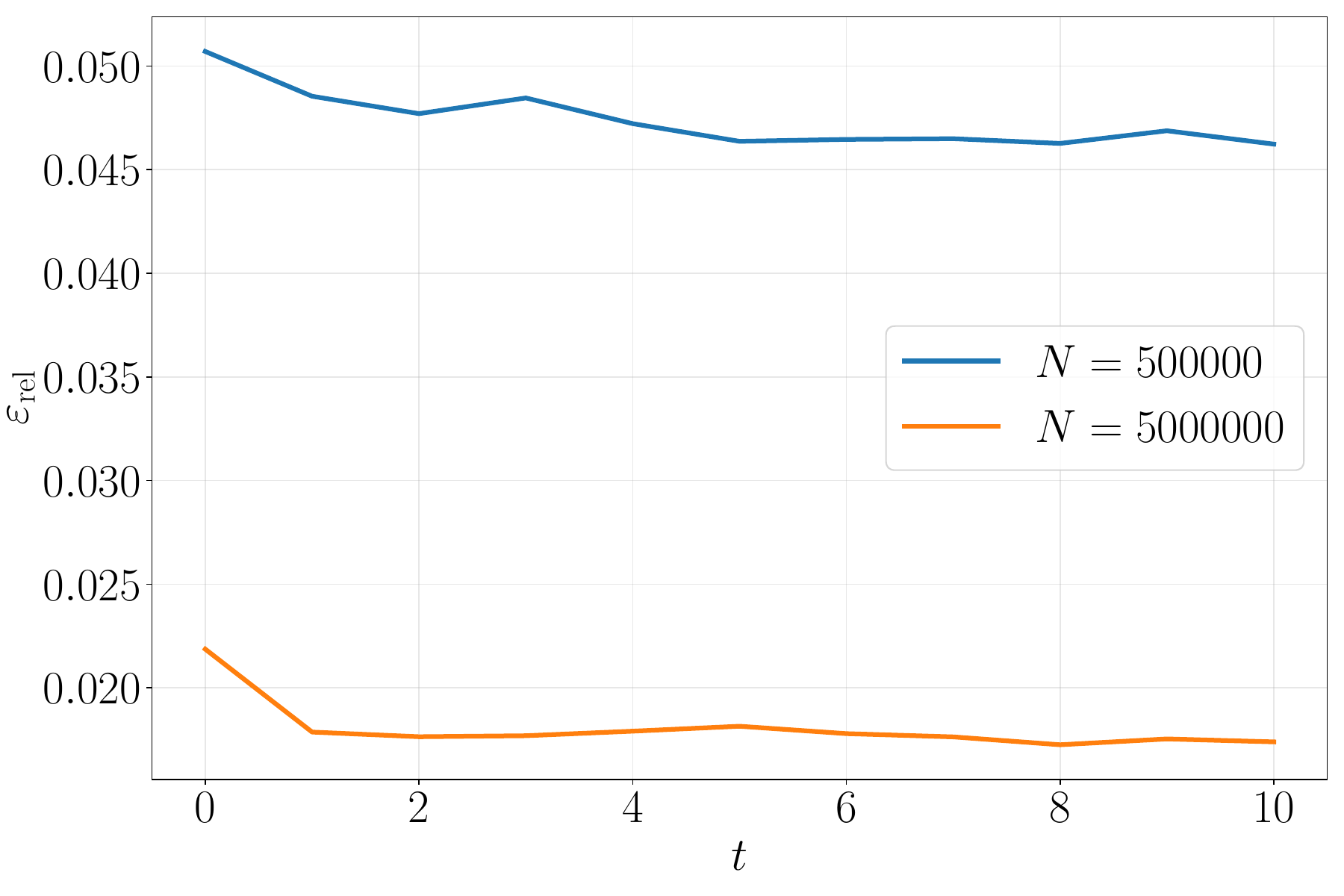}
        \label{fig:3dMaxerror}
    }
    \centering
        \subfigure[Energy]{
        \centering
        \includegraphics[height=60mm,width=75mm]{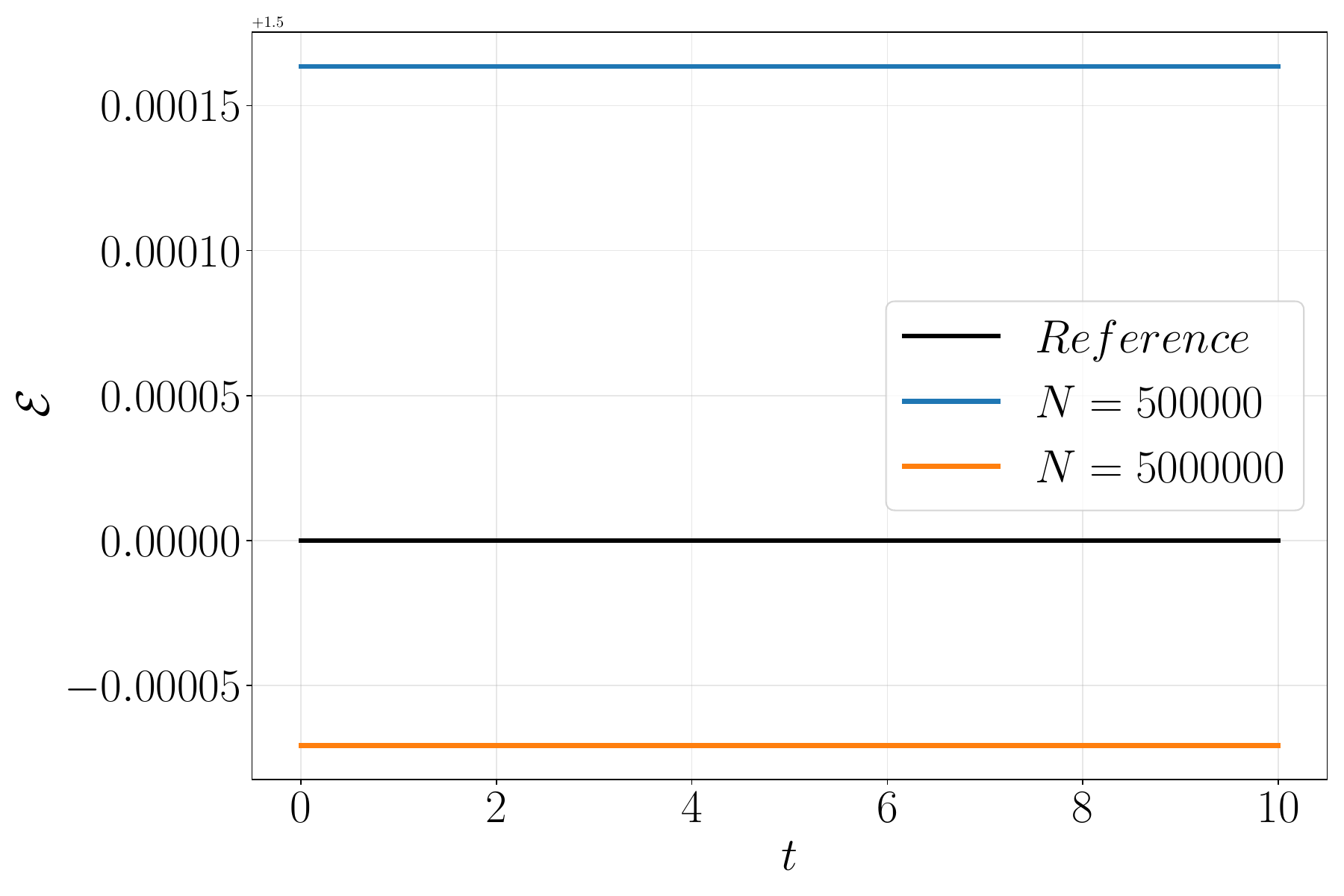}
        \label{fig:3dMaxenergy}
    }
	
		\centering
		\caption{Time evolution of relative $L_2$ error, energy for different $N$.}
		\label{fig:CompardifferN3dMax}
    \end{figure}

    \begin{figure}
        \centering
        \subfigure[$v_x$ cross section at $t=1$]{
        \centering
        \includegraphics[height=60mm,width=75mm]{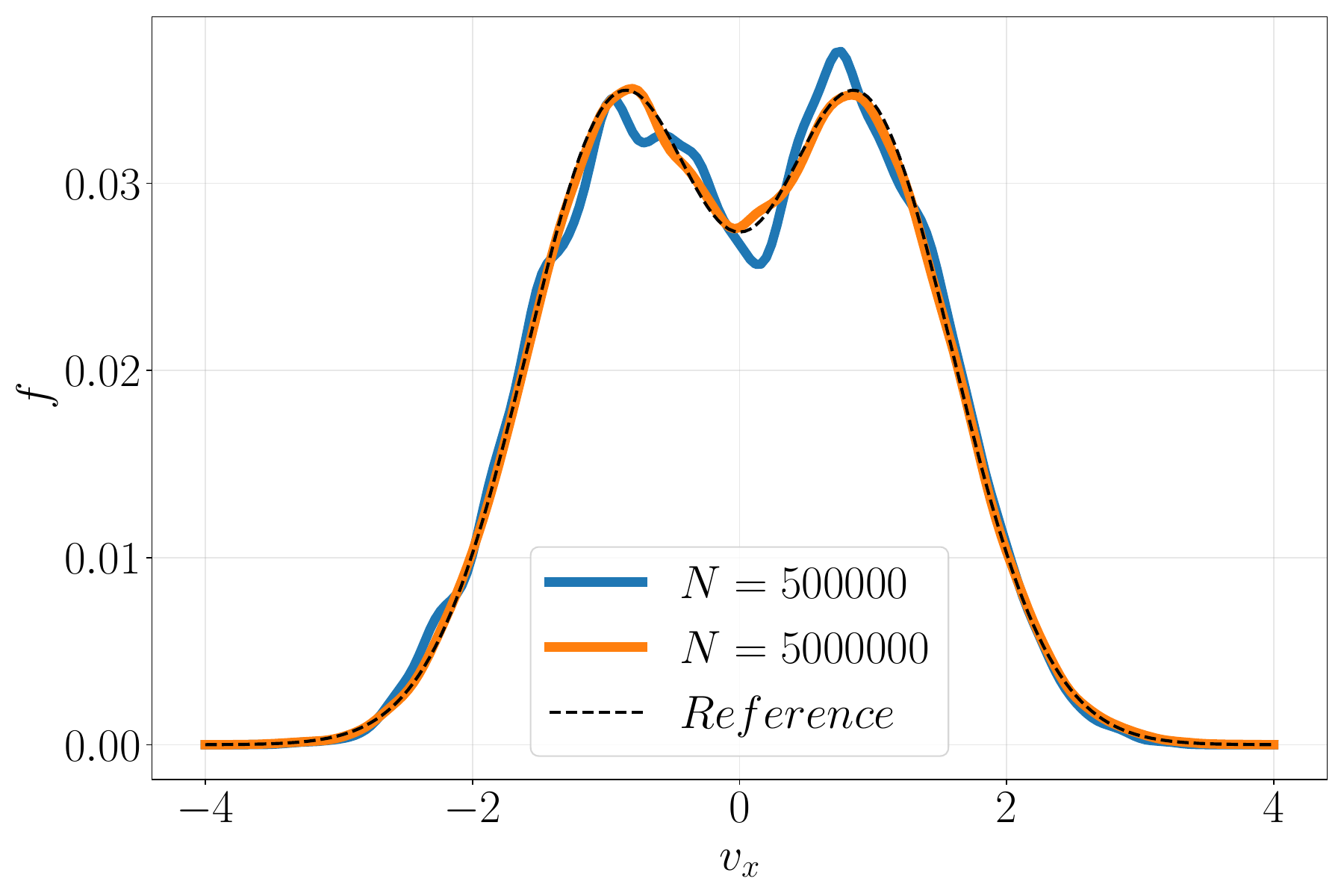}
        \label{fig:3dMaxvxcrosst1}
    }
    \centering
        \subfigure[$v_x$ cross section at $t=10$]{
        \centering
        \includegraphics[height=60mm,width=75mm]{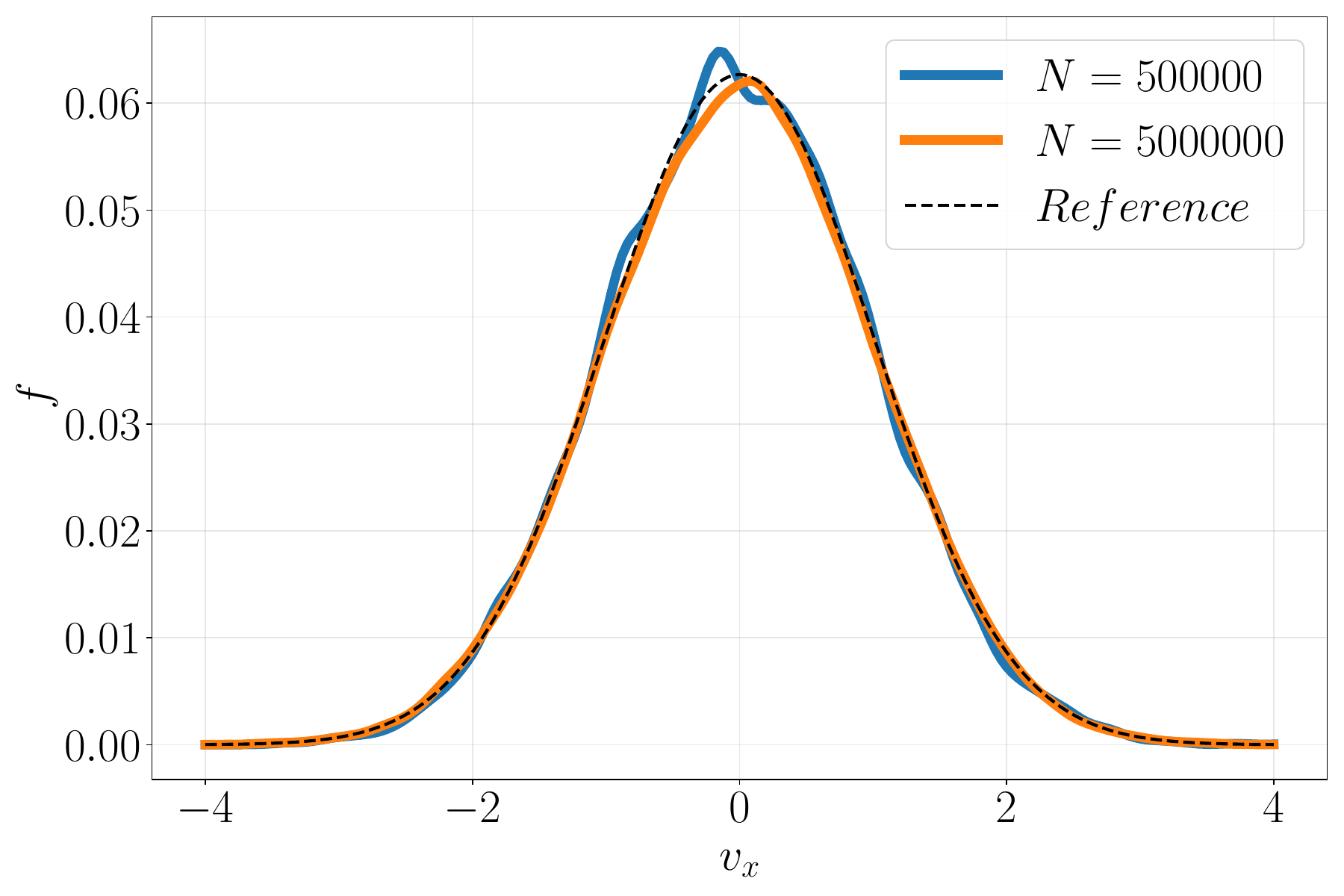}
        \label{fig:3dMaxvxcrosst10}
    }
        \caption{Cross section at different $t$}
        \label{fig:3dMaxcross}
    \end{figure}

\subsection{3D Maxwell molecules with non smooth initial distribution}\label{subsec:3dMaxnosmooiniti}

Here we adopt the problem setting of Test 3 in \cite{WU201327}.

In this case, the collision kernel is still $$B(|q|,\theta)=\frac{1}{4\pi}$$ and the initial condition is given by $$ f(v, t=0) = \frac{1}{3(2\pi)^{3/2}}\begin{cases}
4 \exp \left( -\frac{|v|^2}{2} \right), & v_1 \geqslant 0, \\
\exp \left( -\frac{v_1^2}{8} - \frac{v_2^2+v_3^2}{2} \right), & v_1 < 0.
\end{cases} ,$$ which is not continuous at $v_1=0$.

We set \(t_0 = 0\), \(t_\text{end} = 10\), and \(\Delta t = 0.1\). The initial particle velocities are sampled independently from the initial distribution. NF we used is still the trained model with $(\alpha,\beta)=(0,0)$ in subsection~\ref{subsec:optimiznf}. We use Algorithm~\ref{Boltzmannsphere} to solve the equation with \(N = 500,000\) and \(N = 5,000,000\). We use $N=10,000,000$ and $\Delta t=0.01$ as the reference solution. And It can be obtained analytically \cite{WU201327} that the evolution of the second- and fourth-order moments is given by

\[
P_{xx} = \int f |v_1|^2 \mathrm{d}v =\frac{2}{3} \exp \left( -\frac{t}{2} \right) + \frac{4}{3}, \quad 
M'_4 = \int f |v|^4 \mathrm{d}v = \frac{22}{3} \exp \left( -\frac{t}{3} \right) + \frac{80}{3}.
\] 
The results are presented in Figure~\ref{fig:CompardifferN3dMaxnoC0} and~\ref{fig:3dMaxnoC0cross}. By this experiment it confirms that our method can capture the evolution including the high moment with high accuracy and conserves energy at complex non-continuous initial distribution.

\begin{figure}[!t]
        \centering
        \subfigure[Relative $L_2$ error]{
        \centering
        \includegraphics[height=60mm,width=75mm]{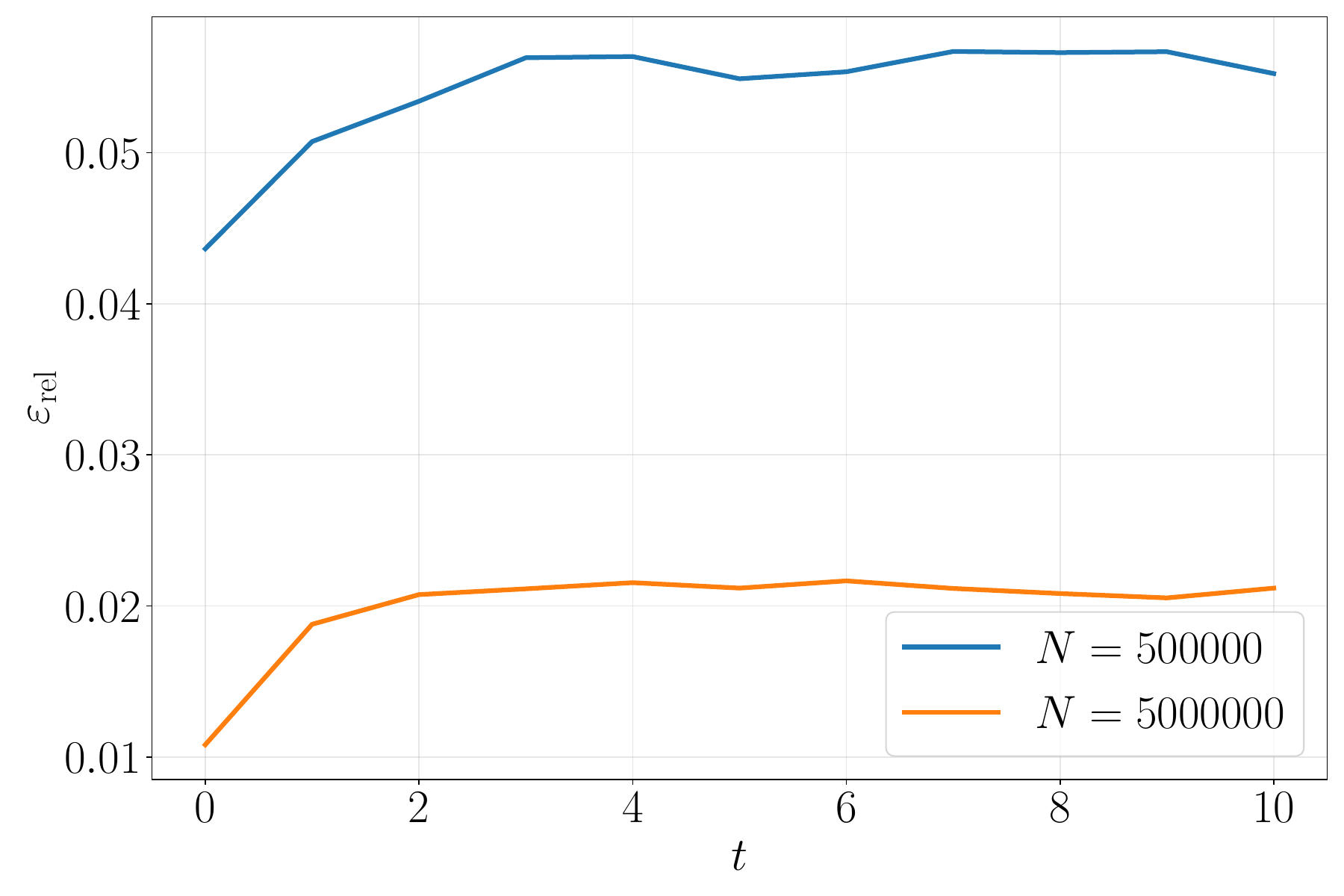}
        \label{fig:3dMaxnoC0error}
    }
    \centering
        \subfigure[Energy]{
        \centering
        \includegraphics[height=60mm,width=75mm]{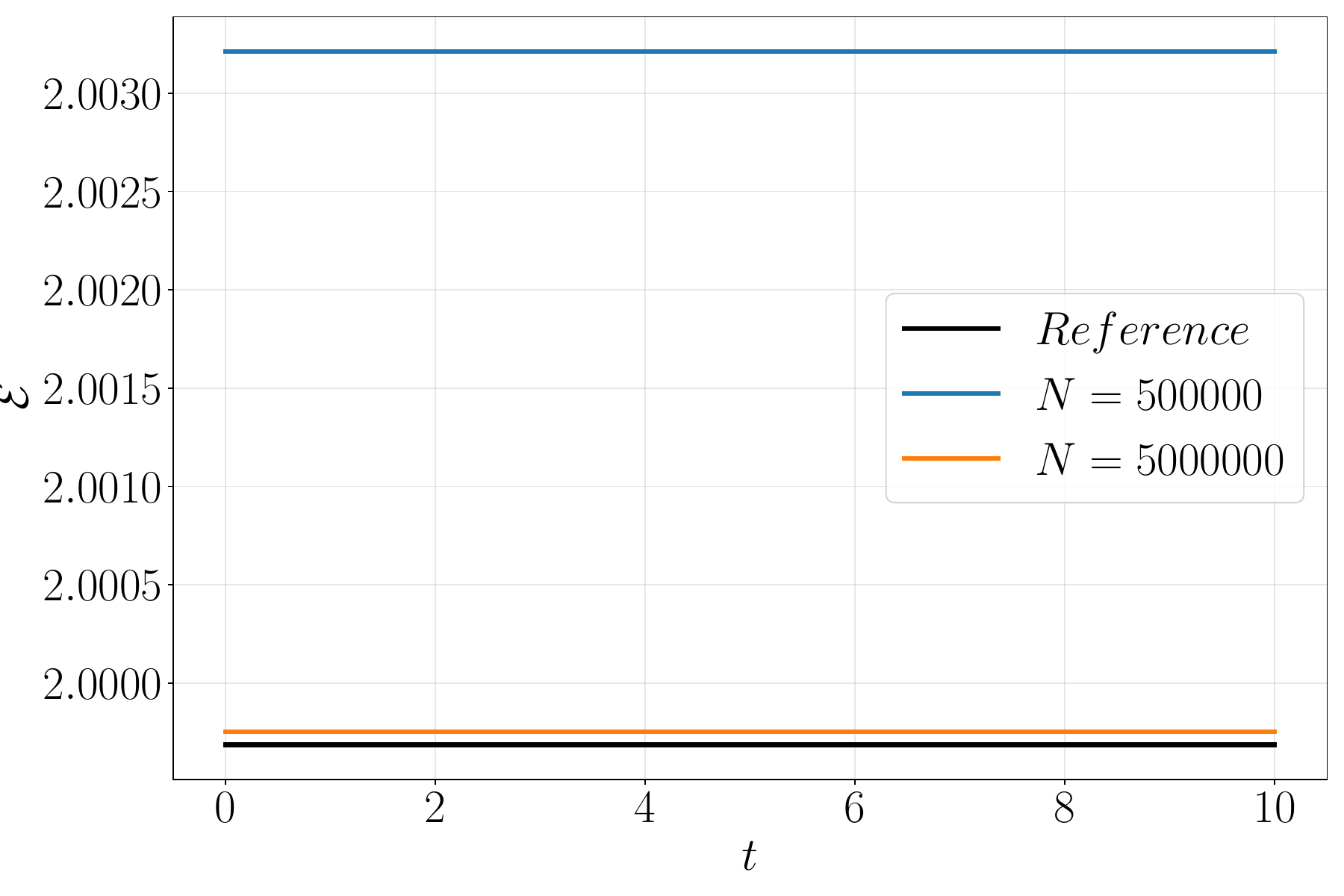}
        \label{fig:3dMaxnoC0energy}
    }
	\\

            \centering
        \subfigure[$P_{xx}$ relative error]{
        \centering
        \includegraphics[height=60mm,width=75mm]{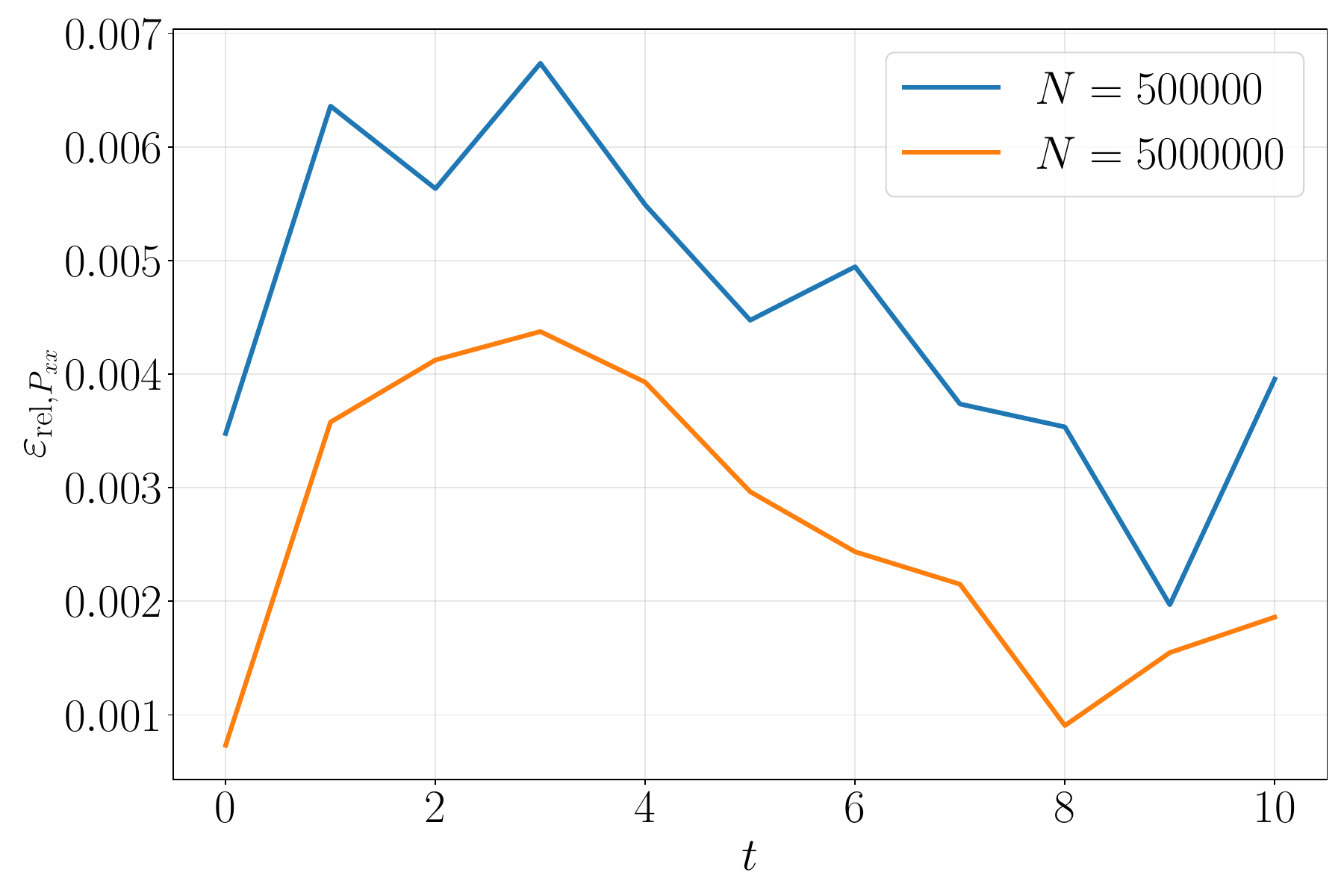}
        \label{fig:3dMaxnoC0Pxx}
    }
    \centering
        \subfigure[$M^{'}_{4}$ relative error]{
        \centering
        \includegraphics[height=60mm,width=75mm]{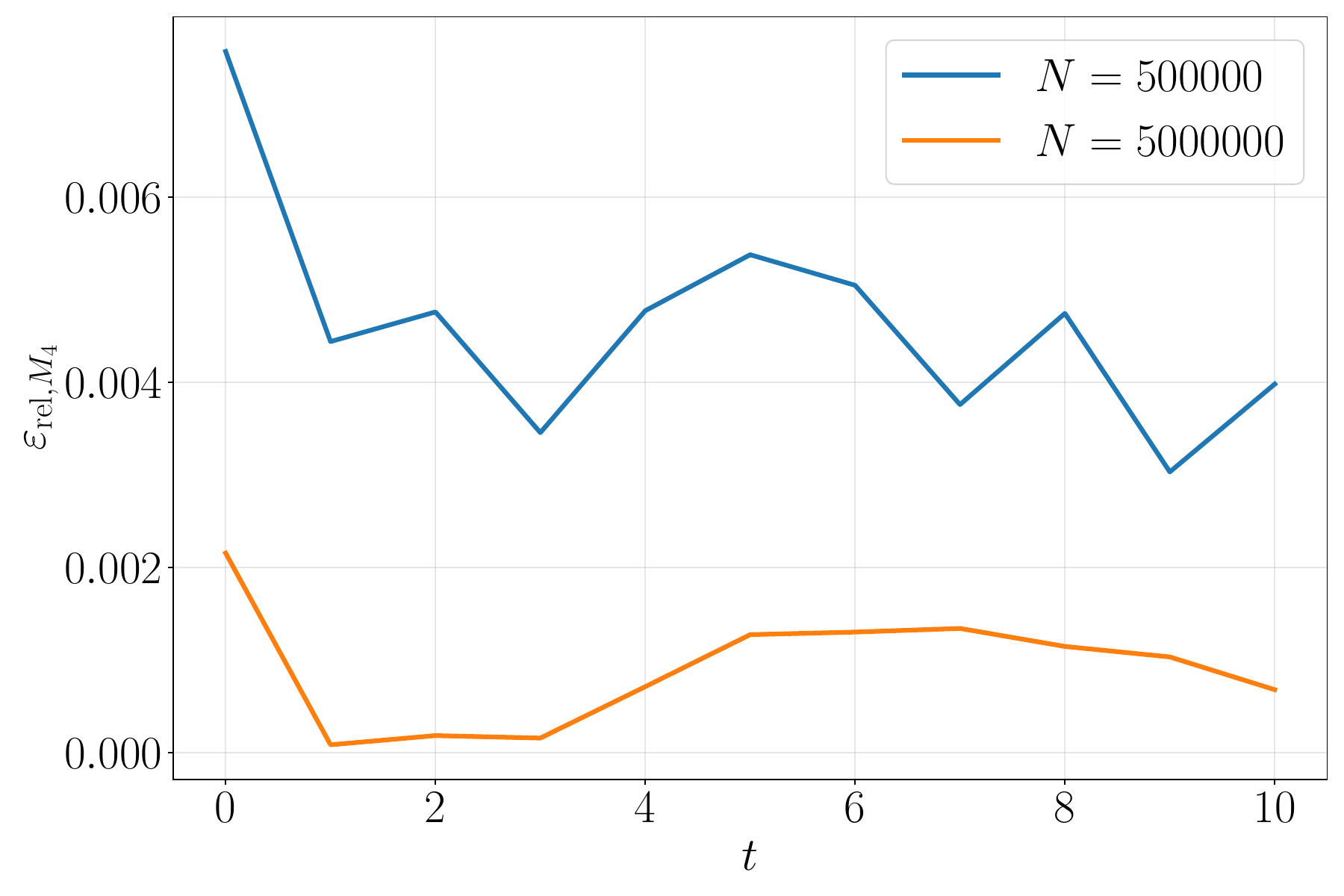}
        \label{fig:3dMaxnoC0M4}
    }
	\centering
	
		\centering
		\caption{Time evolution of relative $L_2$ error, energy, $P_{xx}$ relative error and $M^{'}_{4}$ relative error for different $N$.}
		\label{fig:CompardifferN3dMaxnoC0}
    \end{figure}

    \begin{figure}
        \centering
        \subfigure[$v_x$ cross section at $t=1$]{
        \centering
        \includegraphics[height=60mm,width=75mm]{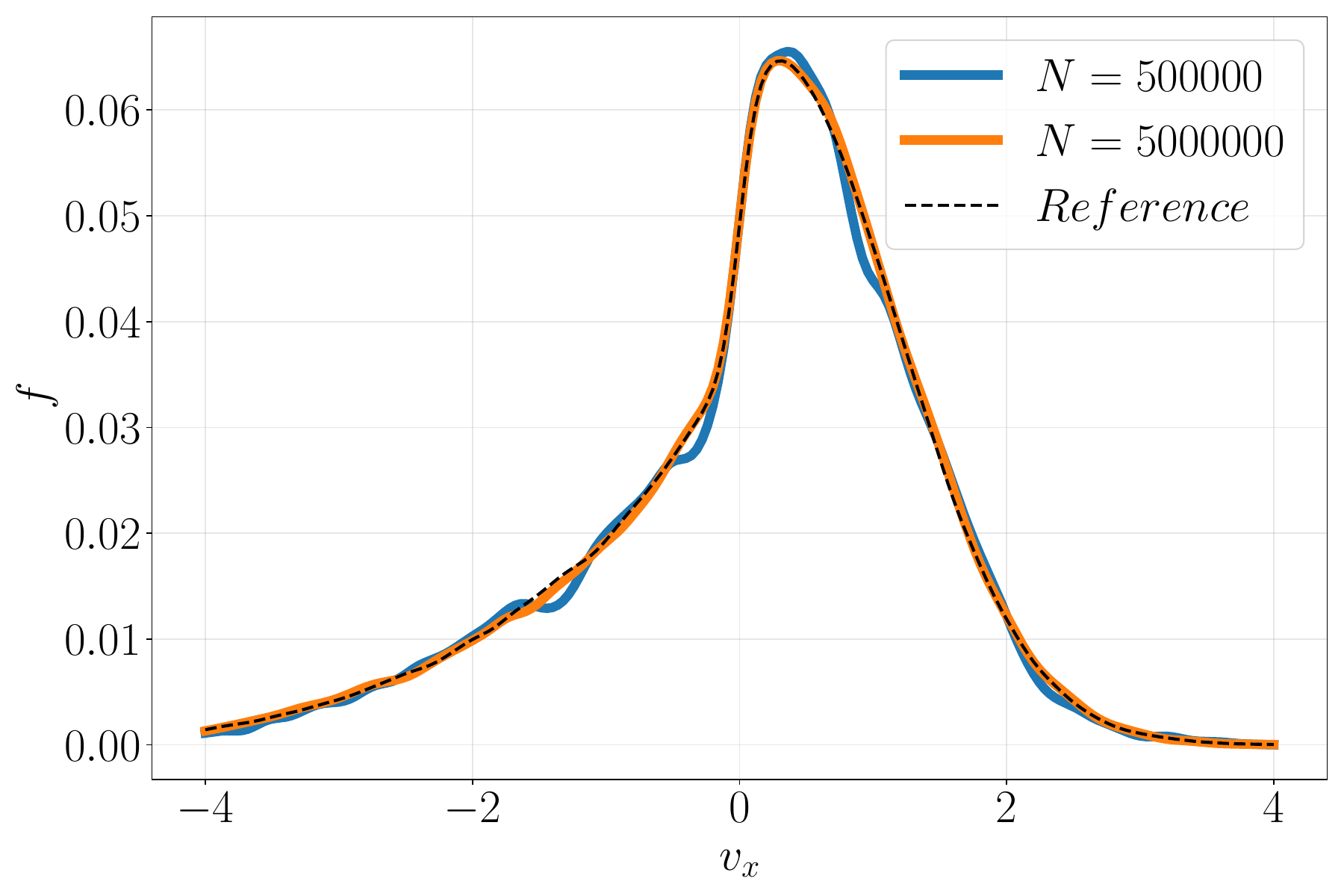}
        \label{fig:3dMaxnoC0vxcrosst1}
    }
    \centering
        \subfigure[$v_x$ cross section at $t=10$]{
        \centering
        \includegraphics[height=60mm,width=75mm]{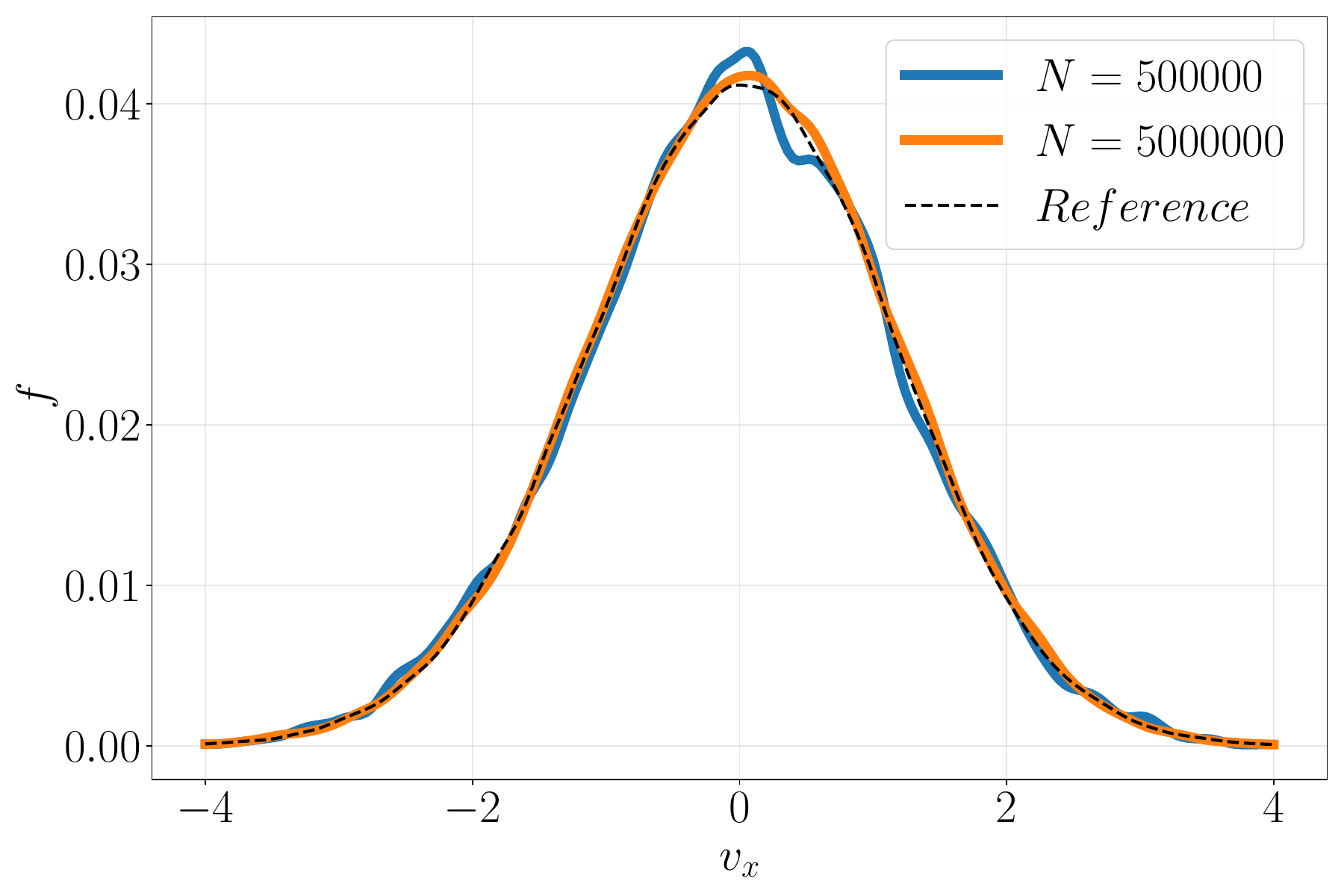}
        \label{fig:3dMaxnoC0vxcrosst10}
    }
        \caption{Cross section at different $t$}
        \label{fig:3dMaxnoC0cross}
    \end{figure}

\subsection{3D anisotropic solution for Ar}\label{subsec:3dAr}

In this example, we demonstrate that our method can effectively solve Boltzmann equations with complex kernel, such as that for argon.

In this case, the collision kernel turns to $$B(|q|,\theta)=\frac{1
}{4\pi}\frac{\beta+2}{2}|q|^{\alpha}\cos^{\beta}(\frac{\theta}{2})$$ with $\alpha=0.384,\beta=1.3248$, and the initial condition is given by
$$ f(0,v)=\frac{1}{2(2\pi)^{1.5}}\Big[2^{1.5}\exp\Big(-(v-u_{1})^{2}\Big)+\exp\Big(-\frac{1}{2}(v-u_{2})^{2}\Big)\Big)\Big],\quad$$ $$  u_{1}=(2, 0, 0), u_{2}=(-1, 0, 0).$$

We set \(t_0 = 0\), \(t_\text{end} = 10\), and \(\Delta t = 0.1\). The initial particle velocities are sampled independently from the initial distribution. NF we used is the trained model with $(\alpha,\beta)=(0.384,1.3248)$ in subsection~\ref{subsec:optimiznf}. We use Algorithm~\ref{Boltzmannsphere} to solve the equation with \(N = 500,000\) and \(N = 5,000,000\). We use $N=10,000,000$ and $\Delta t=0.01$ as the reference solution .

The results are presented in Figure~\ref{fig:CompardifferN3dAr} and~\ref{fig:3dArcross}. Through this experiment it confirms that our method can capture the evolution of particle collision with high accuracy and conserves energy with complex form of kernel.

\begin{figure}[!t]
        \centering
        \subfigure[Relative $L_2$ error]{
        \centering
        \includegraphics[height=60mm,width=75mm]{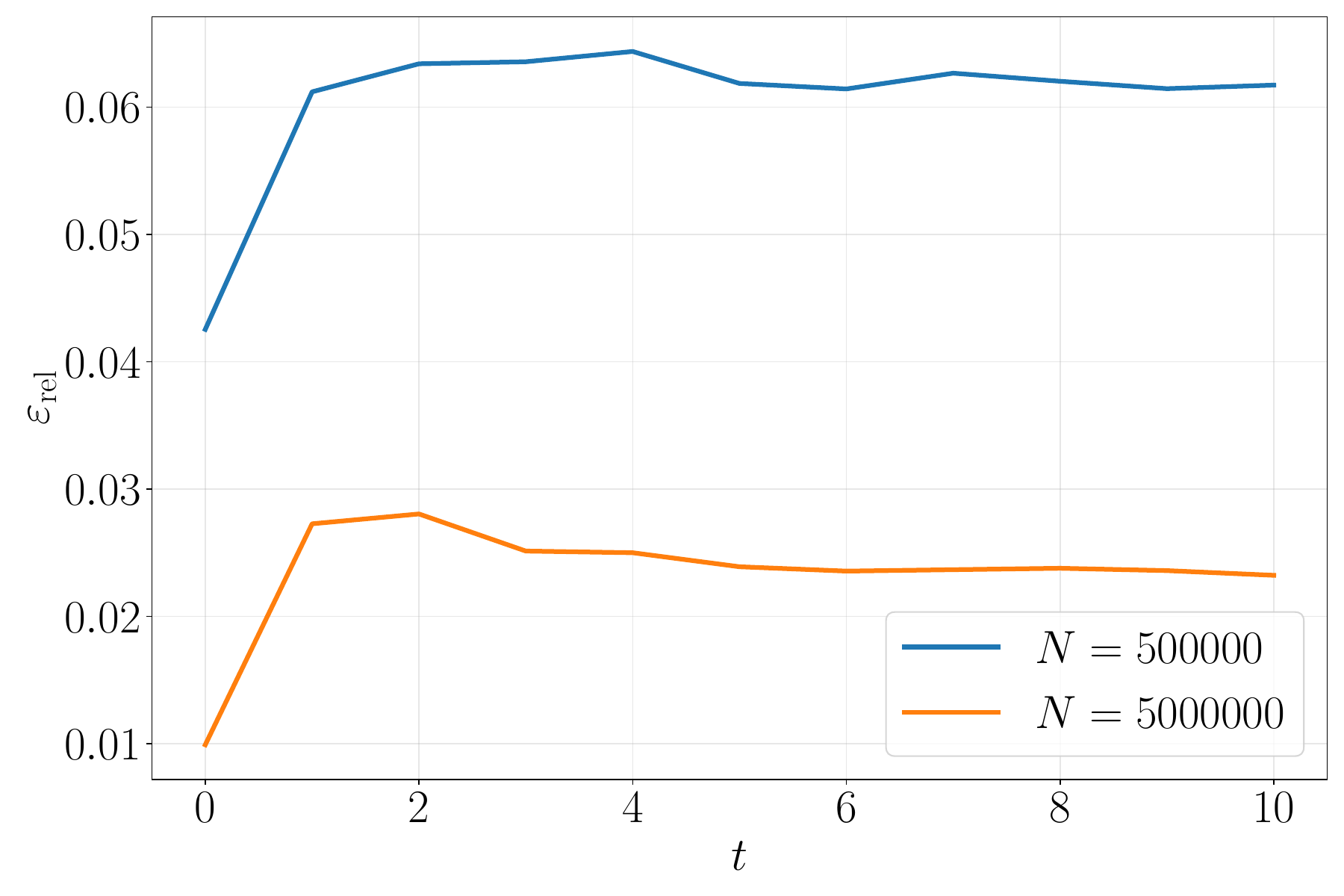}
        \label{fig:3dArerror}
    }
    \centering
        \subfigure[Energy]{
        \centering
        \includegraphics[height=60mm,width=75mm]{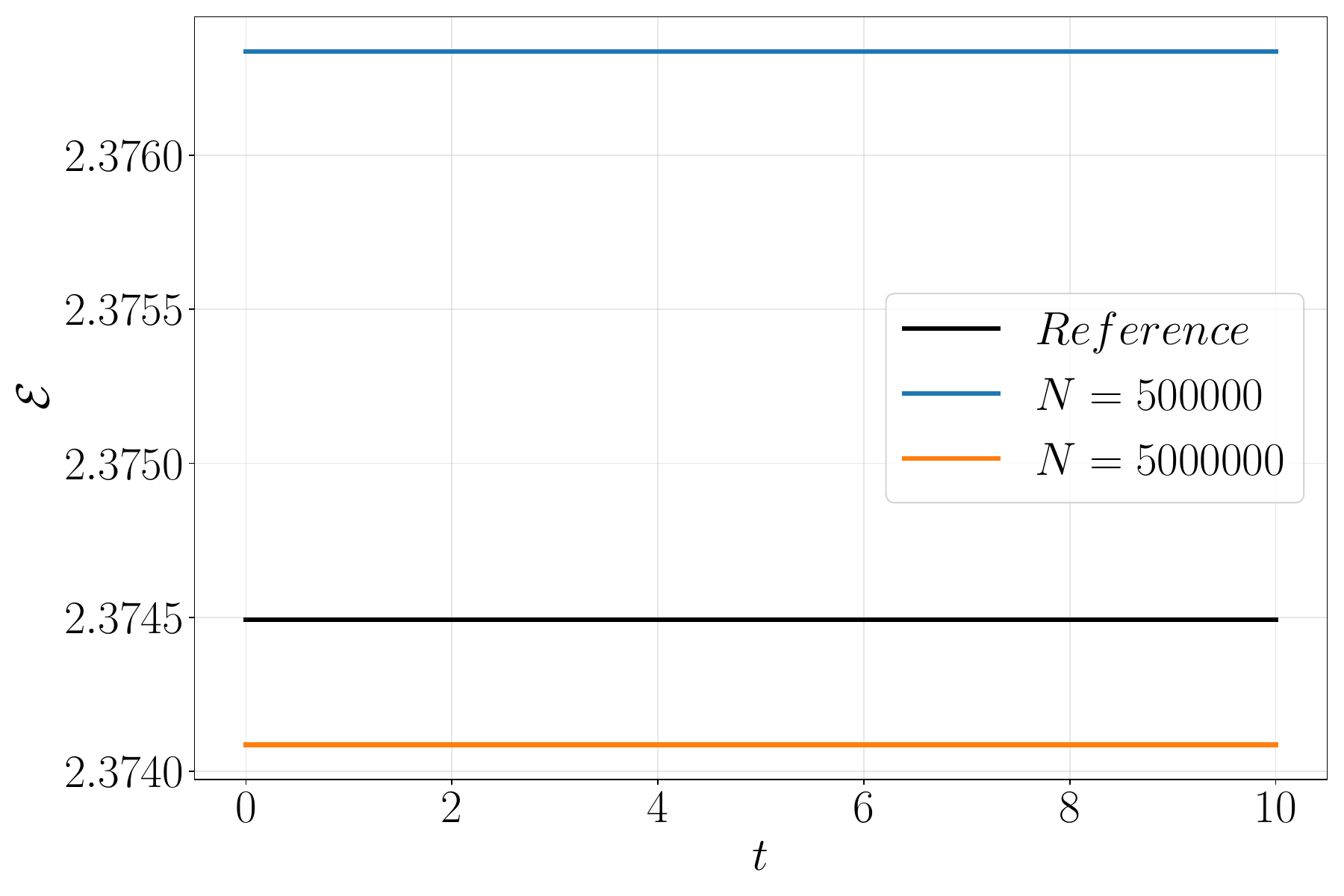}
        \label{fig:3dArenergy}
    }
		\centering
		\caption{Time evolution of relative $L_2$ error, energy for different $N$.}
		\label{fig:CompardifferN3dAr}
    \end{figure}

    \begin{figure}
        \centering
        \subfigure[$v_x$ cross section at $t=1$]{
        \centering
        \includegraphics[height=60mm,width=75mm]{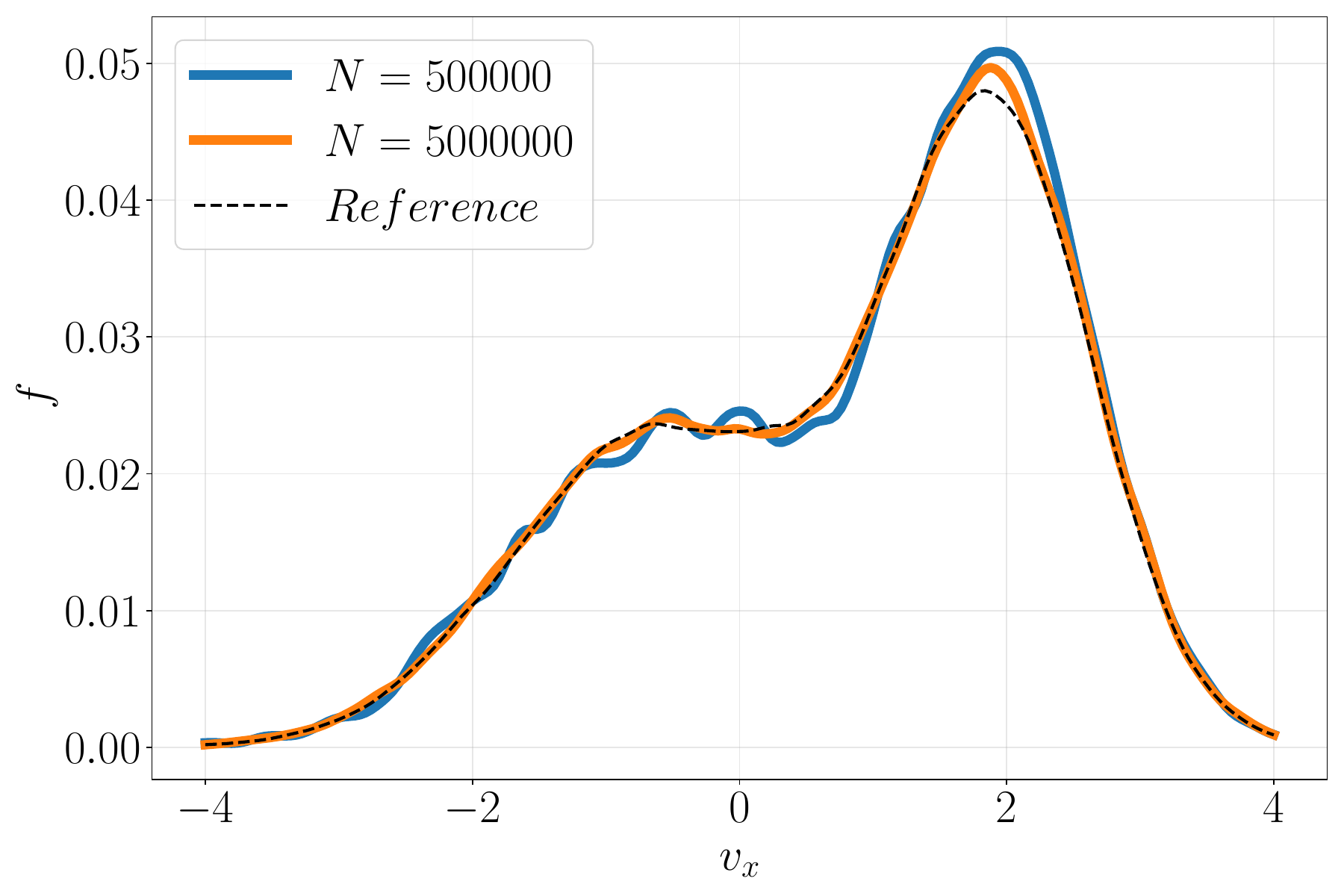}
        \label{fig:3dArvxcrosst1}
    }
    \centering
        \subfigure[$v_x$ cross section at $t=10$]{
        \centering
        \includegraphics[height=60mm,width=75mm]{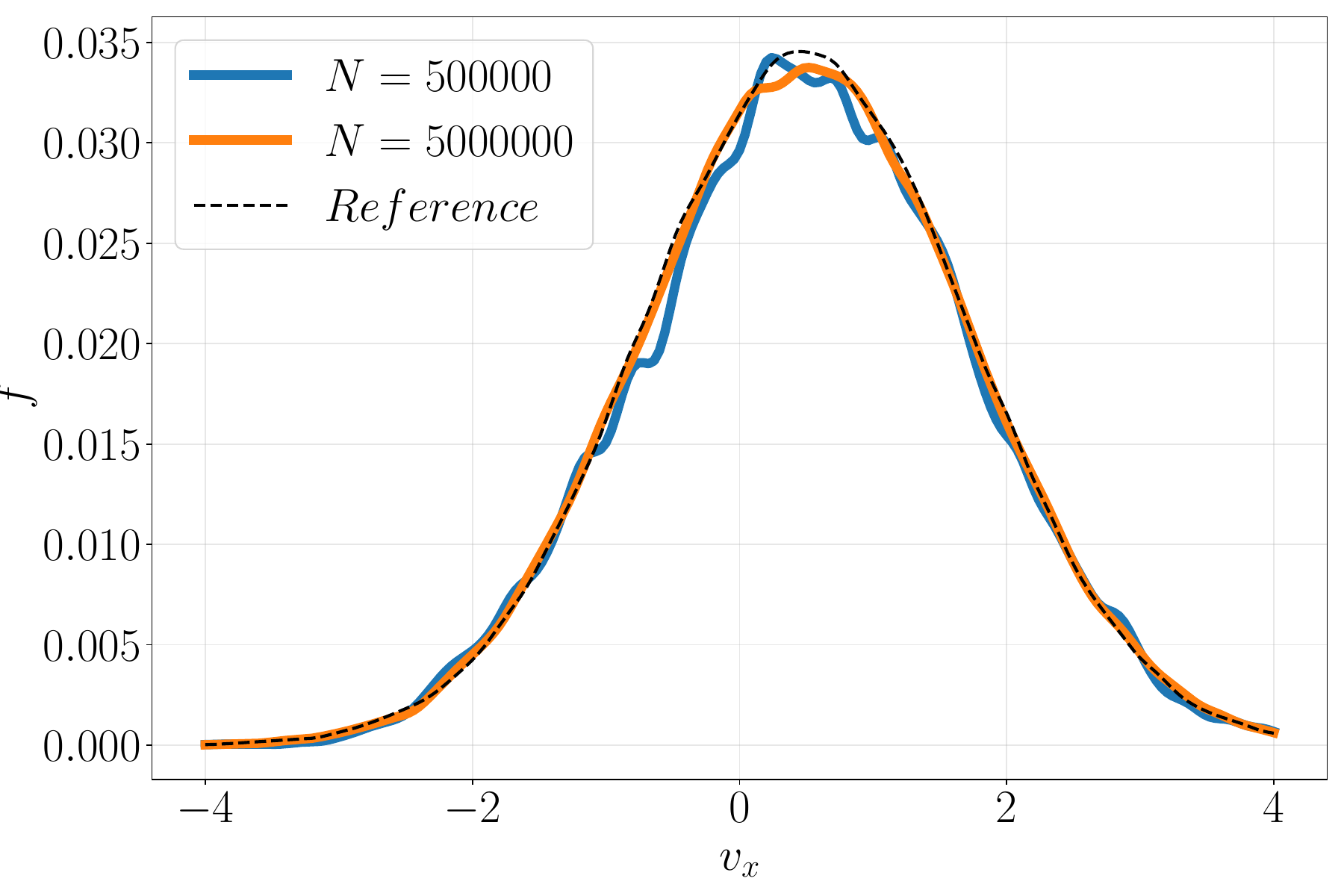}
        \label{fig:3dArvxcrosst=100}
    }
        \caption{Cross section at different $t$}
        \label{fig:3dArcross}
    \end{figure}

\subsection{A Riemann problem}\label{subsec:Riemann}

Eventually we demonstrate the effectiveness of our method when extended to spatially nonhomogeneous cases. 

For simplicity, we assume that particles are uniformly distributed in the second and third spatial dimension, leading to homogeneity in these directions. Hence, we only need to consider the first spatial dimension. We take the first spatial domain as $[0,1]$ with reflective boundary condition.

In this scenario, the collision kernel is $$B(|q|,\theta)=\frac{1}{4\pi},$$and the initial condition is a local Maxwellian distribution computed from the following macroscopic quantities
\begin{equation*}
\begin{cases}
(\rho_l, u_l, T_l) = (1, 0, 1), & \text{if } 0 \leqslant x \leqslant 0.5, \\
(\rho_r, u_r, T_r) = (0.125, 0, 0.25), & \text{if } 0.5 < x \leqslant 1.
\end{cases} 
\end{equation*}

We set \(t_0 = 0\), \(t_\text{end} = 0.25\), and \(\Delta t = 0.005\). The initial particle velocities and positions are sampled independently from the initial distribution. NF we used is still the trained model with $(\alpha,\beta)=(0,0)$ in subsection~\ref{subsec:optimiznf}. We use Algorithm~\ref{Boltzmannsphere} and PIC algorithm to solve the equation with \(N = 500,000\) and \(N = 5,000,000\), where we divide $n_x=100$ cells in $[0,1]$. We use $N=10,000,000$ and $\Delta t=0.001$ as the reference solution.

The results are presented in Figure~\ref{fig:CompardifferN1D3V} and~\ref{fig:1D3Vmeasure}, confirming that our method can capture the evolution with high accuracy and conserves energy in spatially nonhomogeneous cases as well.

\begin{figure}[!t]
        \centering
        \subfigure[Relative $L_2$ error]{
        \centering
        \includegraphics[height=60mm,width=75mm]{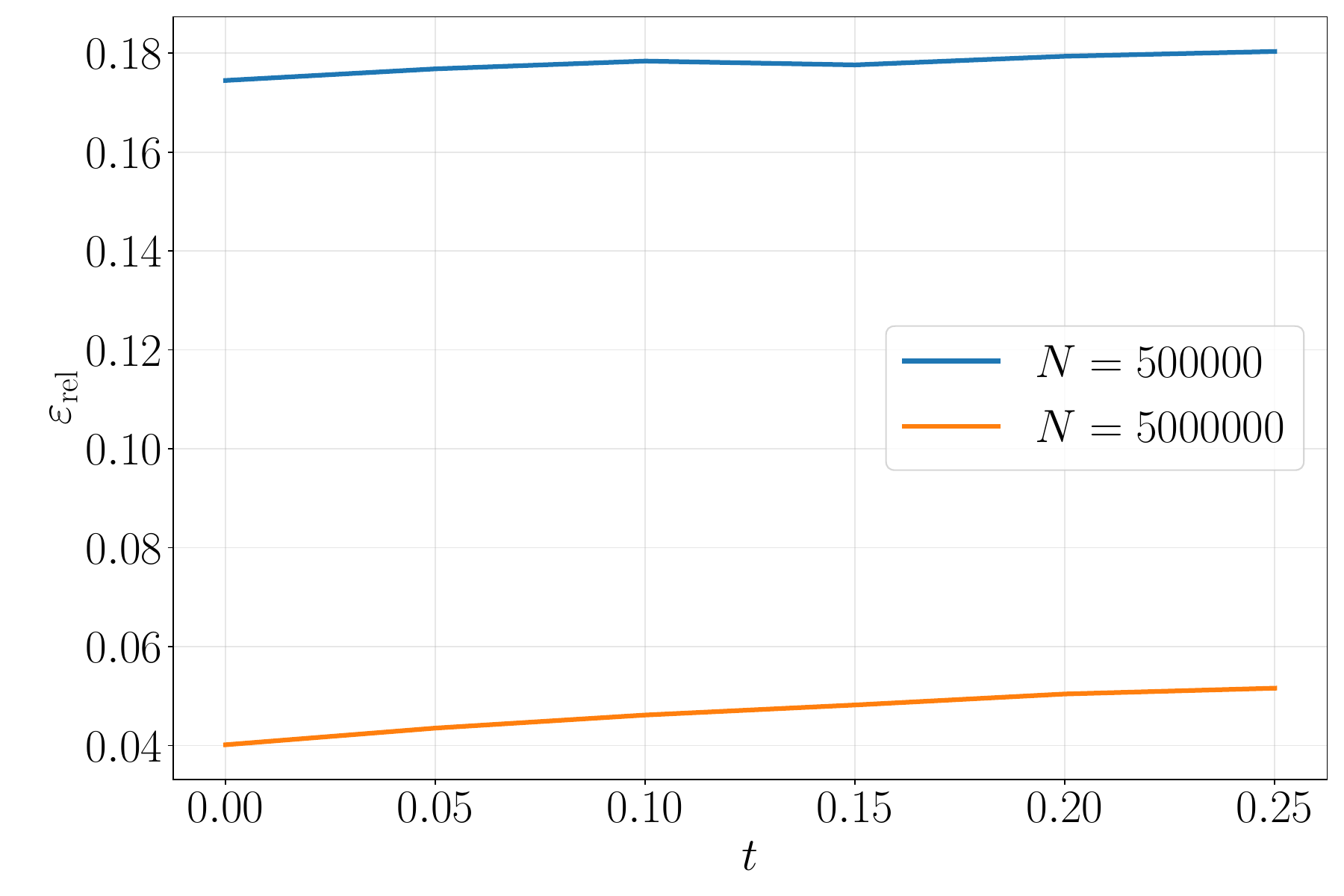}
        \label{fig:1D3Verror}
    }
    \centering
        \subfigure[Energy]{
        \centering
        \includegraphics[height=60mm,width=75mm]{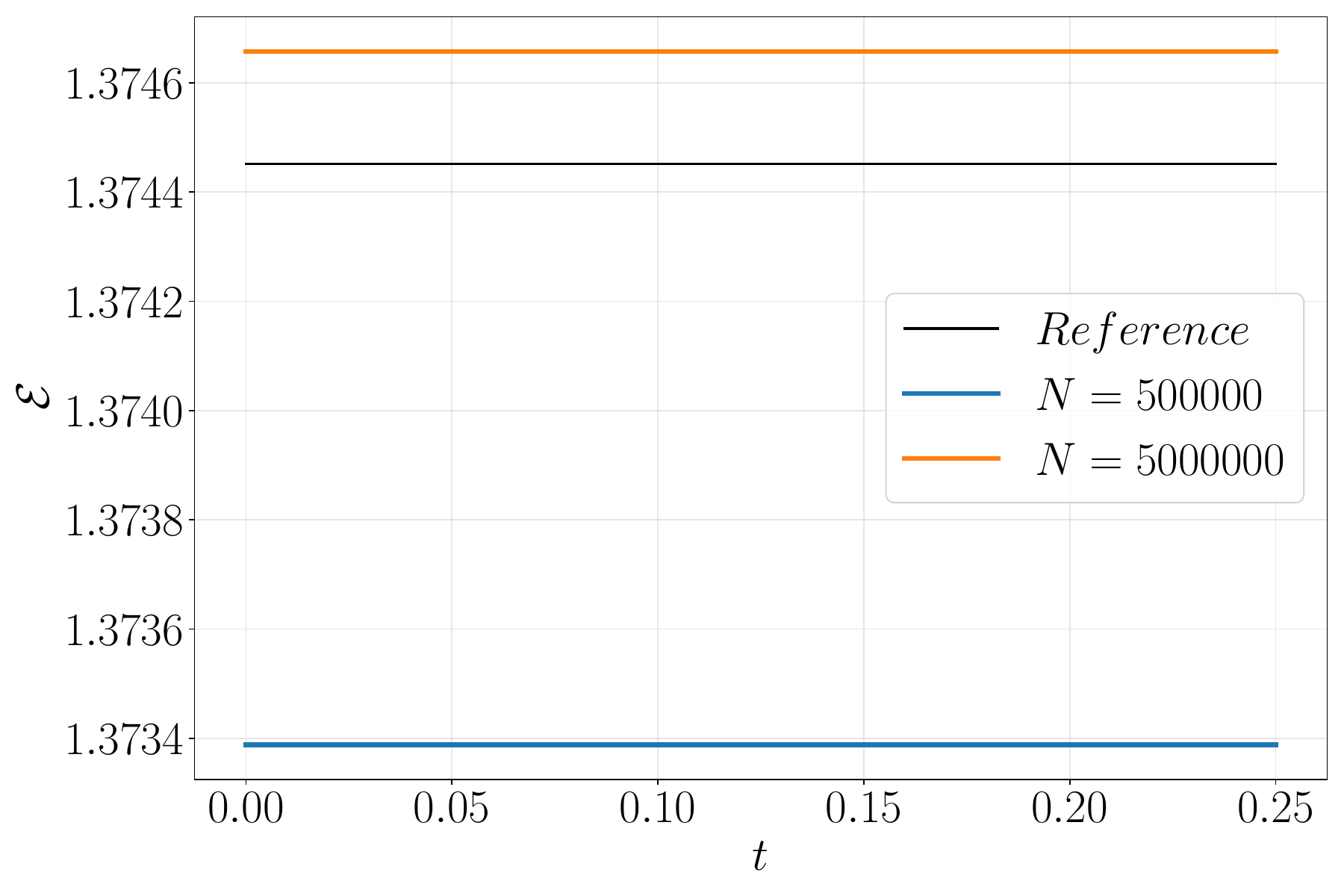}
        \label{fig:1d3venergy}
    }
		\centering
		\caption{Time evolution of relative $L_2$ error, energy for different $N$.}
		\label{fig:CompardifferN1D3V}
    \end{figure}

    \begin{figure}
        \centering
        \subfigure[Density $\rho$]{
        \centering
        \includegraphics[height=60mm,width=75mm]{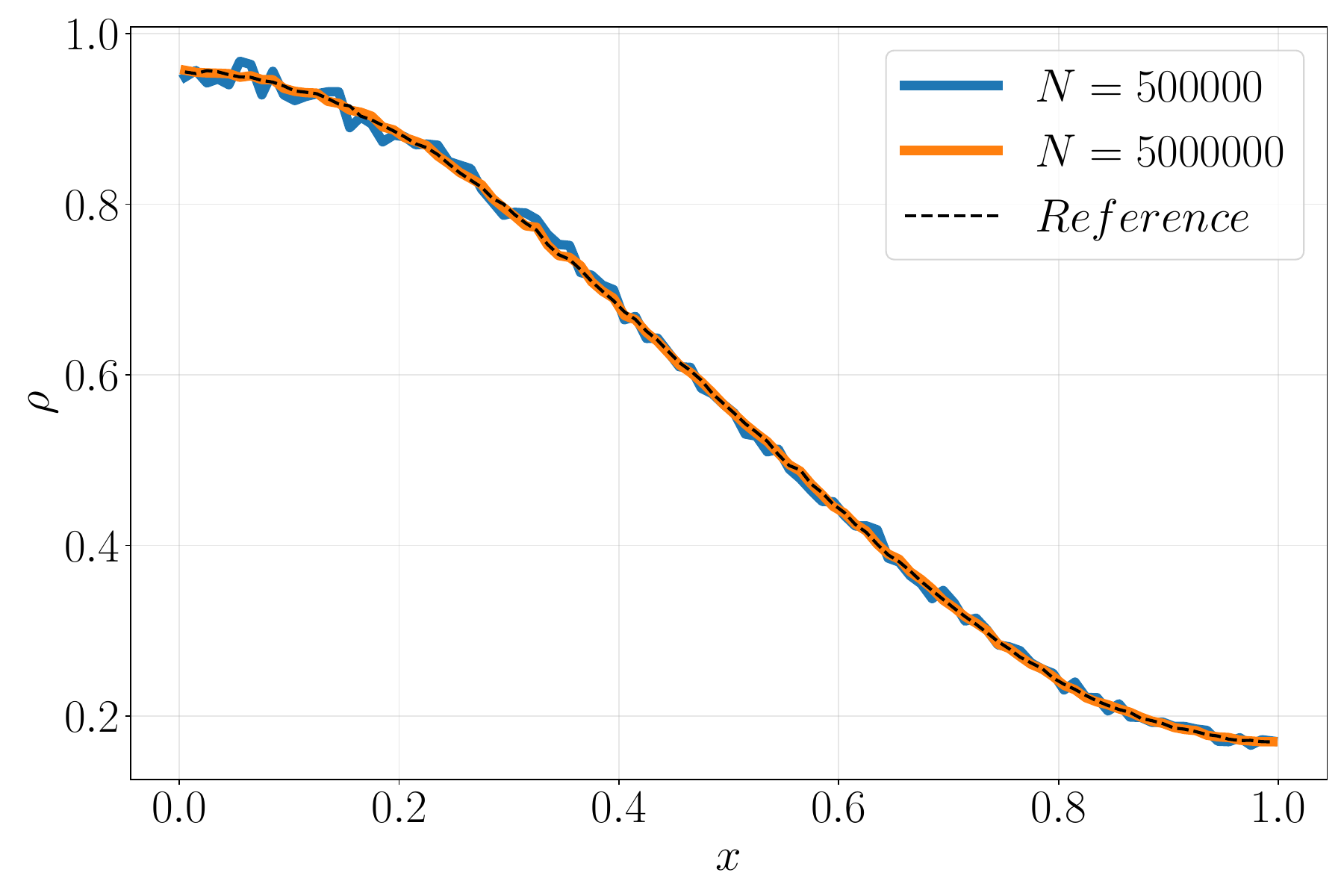}
        \label{fig:1D3Vdensity}
    }
    \centering
        \subfigure[Mean velocity $u$]{
        \centering
        \includegraphics[height=60mm,width=75mm]{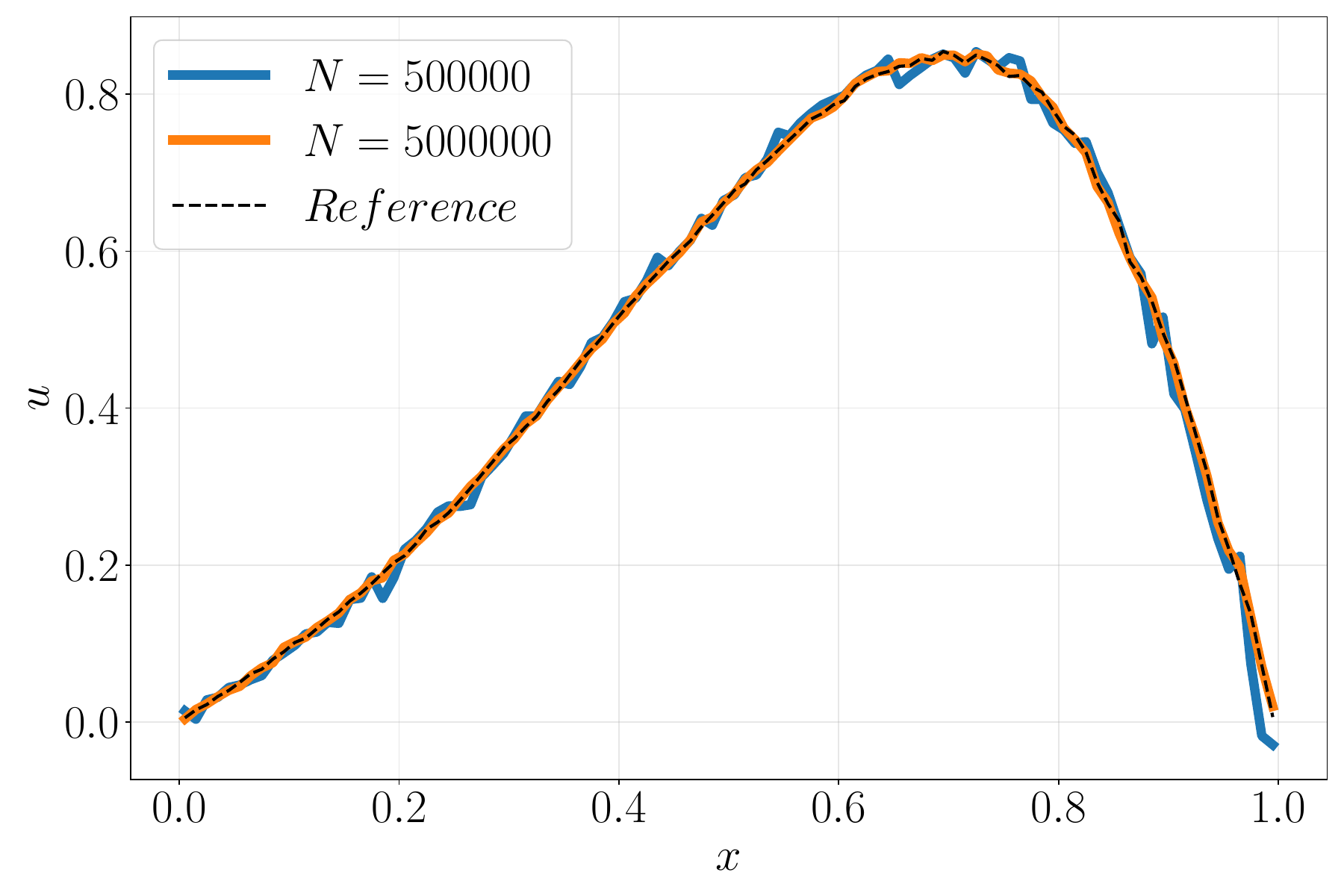}
        \label{fig:1D3Vvelocity}
    }
    \centering
        \subfigure[Temperature $T$]{
        \centering
        \includegraphics[height=60mm,width=75mm]{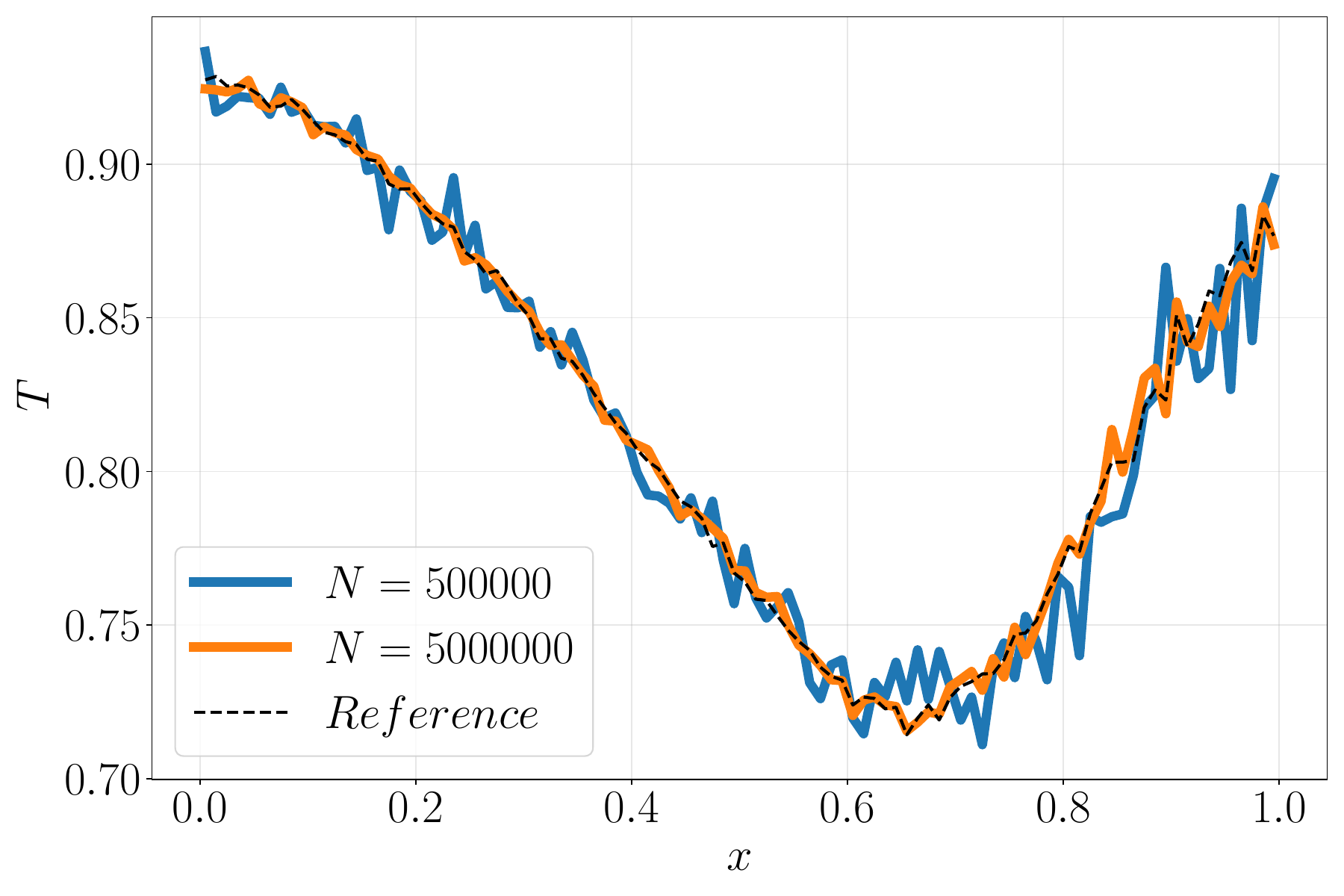}
        \label{fig:1d3Vtemperature}
    }
    \centering
            \subfigure[Density of phase space $(x,v_x)$ with $N=5,000,000$]{
        \centering
        \includegraphics[height=60mm,width=75mm]{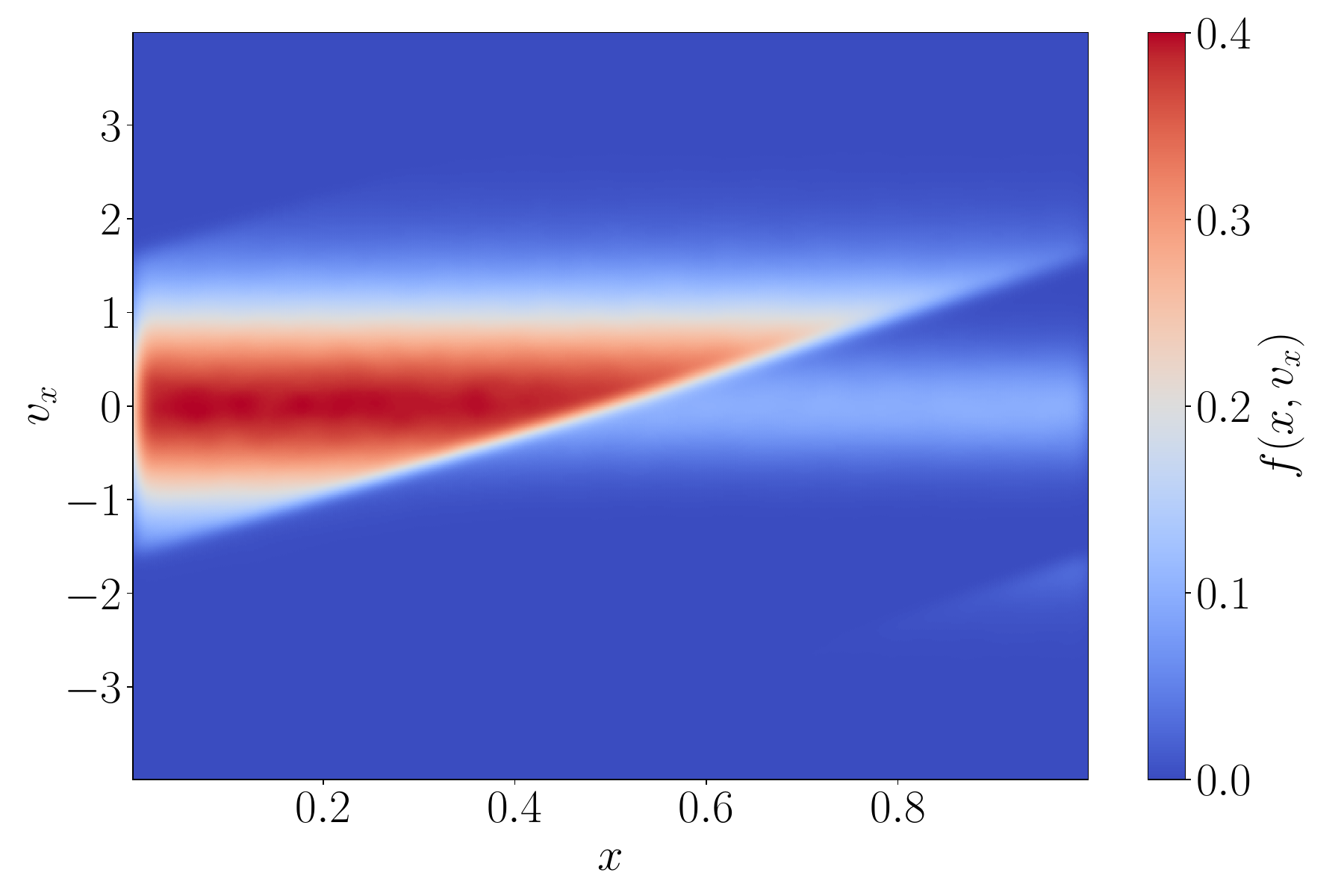}
        \label{fig:1D3Vphase5000000}
    }
        \caption{Density $\rho$, mean velocity $u$, temperature $T$ and phase space density $f(x,v_x)$ at $t=0.25$}
        \label{fig:1D3Vmeasure}
    \end{figure}

\bibliographystyle{plain}
\bibliography{refs}

\appendix
\section{Architecture of the Conditional Monotone Flow}
\label{app:nf}

Here we give the detailed architecture of the conditional monotone flow used in subsection~\ref{subsec:NF}. 

The condition is $c=(|q|,\Delta t) $ and it is first mapped by a condition encoder to a condition embedding $e=e(c)$ and subsequently passed through a FiLM parameter network 
\cite{perez2018film} $e_{F}=\text{FiLM}(e)$ to generate the parameters of conditional flow layers.

The full one-dimensional flow is written as the following composition:
\begin{equation}
    y=T_\Theta(z;c)=A\circ S_{2}\circ H_L \circ\cdots\circ H_1\circ S_{1}(z;e).
\end{equation}
Here \(z\sim p_0\) is the initial sample, \(y\) is the generated sample and \(\Theta\) denotes all trainable parameters.

The two Sinh-Arcsinh-type transformations are used before and after the stacked monotone layers. Motivated by the Sinh-Arcsinh transformation \cite{jones2009sinh}, we use the following conditional reparameterized form.
For \(m=1,2\), define
\begin{equation*}
    \gamma_m(u;e)=\frac{   \operatorname{asinh}(u)+\epsilon_m(e)}{\delta_m(e)},\quad\delta_m(e)>0, \qquad S_m(u;e)=s_m(e)+\exp(r_m(e))\sinh(\gamma_m(u;e)),
\end{equation*}where \(s_m(e)\), \(r_m(e)\), \(\epsilon_m(e)\) and \(\delta_m(e)\) are generated from the condition embedding. Here \(s_m(e)\) controls the shift, \(r_m(e)\) controls the logarithmic scale, \(\epsilon_m(e)\) controls skewness and \(\delta_m(e)\) controls tail behavior. The log-Jacobian is explicit:
\begin{equation}
    \log\left|\frac{\partial S_m}{\partial u}(u;e)\right|=r_m(e)+\log\cosh(\xi_m(u;e))-\log\delta_m(e)-\frac12\log(1+u^2).
\end{equation}

The main part of the flow consists of \(L\) strictly monotone tanh residual layers. The \(l\)-th layer is
\begin{equation}
\label{eq:monottanhlayer}
    H_l(u;e)=s_l(e_{F})+\exp(r_l(e_{F}))\left[u+\sum_{i=1}^{M}a_{li}(e_{F})\tanh\bigg(k_{li}(e_{F})u+d_{li}(e_{F})\bigg)\right].
\end{equation}
Here \(u\) is the scalar input of the layer, \(s_l(e_{F})\) is a conditional shift, \(r_l(e_{F})\) is a conditional log-scale, \(a_{l i}(e_{F})\) is the amplitude of the \(i\)-th tanh component, \(k_{li}(e_{F})\) is its slope parameter and \(d_{li}(e_{F})\) is its bias. Moreover we impose \(a_{li}(e_{F})\ge0,k_{li}(e_{F})\ge0\). Therefore
\begin{equation}
    \frac{\partial H_l}{\partial u}(u;e)=\exp(r_l(e_{F}))\left[1+\sum_{i=1}^{M}a_{li}(e_{F})k_{li}(e_{F})\operatorname{sech}^2\bigg(k_{li}(e_{F})u+d_{li}(e_{F})\bigg)\right]>0.
\end{equation}
Thus each \(H_l\) is strictly increasing, and hence invertible.

Finally, the conditional global affine head is
\begin{equation}
\label{eq:globalaffine}
    A(u;e)=s(e)+\exp(r(e))u,
\end{equation}
where \(s(e)\) and \(\exp(r(e))\) capture the main condition-dependent changes in location and log-scale.

Since every component is a one-dimensional monotone transformation, the full map \(T_\Theta(\cdot;c)\) is invertible. Its log-Jacobian is obtained by summing the log-Jacobians of all components in the composition.

\section{Additional results for optimization and NF learning}\label{app:optimizNF}

In subsection~\ref{subsec:optimiznf} to assess the robustness of the proposed method against noisy measurements, we actually contaminated the training data with three types of Gaussian noise. 

First, additive noise is added to the scattering angle:
\[
    \theta_i^{\mathrm{obs}}
    =
    \theta_i+\varepsilon_i^\theta,
    \qquad
    \varepsilon_i^\theta\sim \mathcal N(0,0.01^2).
\]
Second, multiplicative noise is added to the total cross section:
\[
    \sigma_{\mathrm{total}}^{\mathrm{obs}}(|q_j|)
    =
    \sigma_{\mathrm{total}}(|q_j|)
    \left(1+\varepsilon_j^\sigma\right),
    \qquad
    \varepsilon_j^\sigma\sim \mathcal N(0,0.02^2).
\]
Third, the relative speed is perturbed by relative Gaussian noise:
\[
    |q_i|^{\mathrm{obs}}
    =
    |q_i|
    \left(1+\varepsilon_i^q\right),
    \qquad
    \varepsilon_i^q\sim \mathcal N(0,0.01^2).
\]
    
Figure~\ref{fig:twooptimizsupple} gives additional diagnostics for the reconstructed impact parameter $b$. For each VSS test case, we compare the reconstructed and exact values of \(b(|q|,\theta)\), its angular derivative \(\partial_\theta b(|q|,\theta)\) and $b(|q|,\theta)\sin\theta$.
The first row of each subfigure shows these three quantities at several representative relative speeds, and the second row shows the corresponding absolute errors. The agreement confirms that the optimized impact-parameter map recovers not only the eigenvalues reported in the main text, but also the underlying angular structure used to reconstruct the collision kernel.

\begin{figure}[!t]
    \centering
    \subfigure[$\alpha=0,\ \beta=0$]{
        \centering
        \includegraphics[height=90mm,width=160mm]{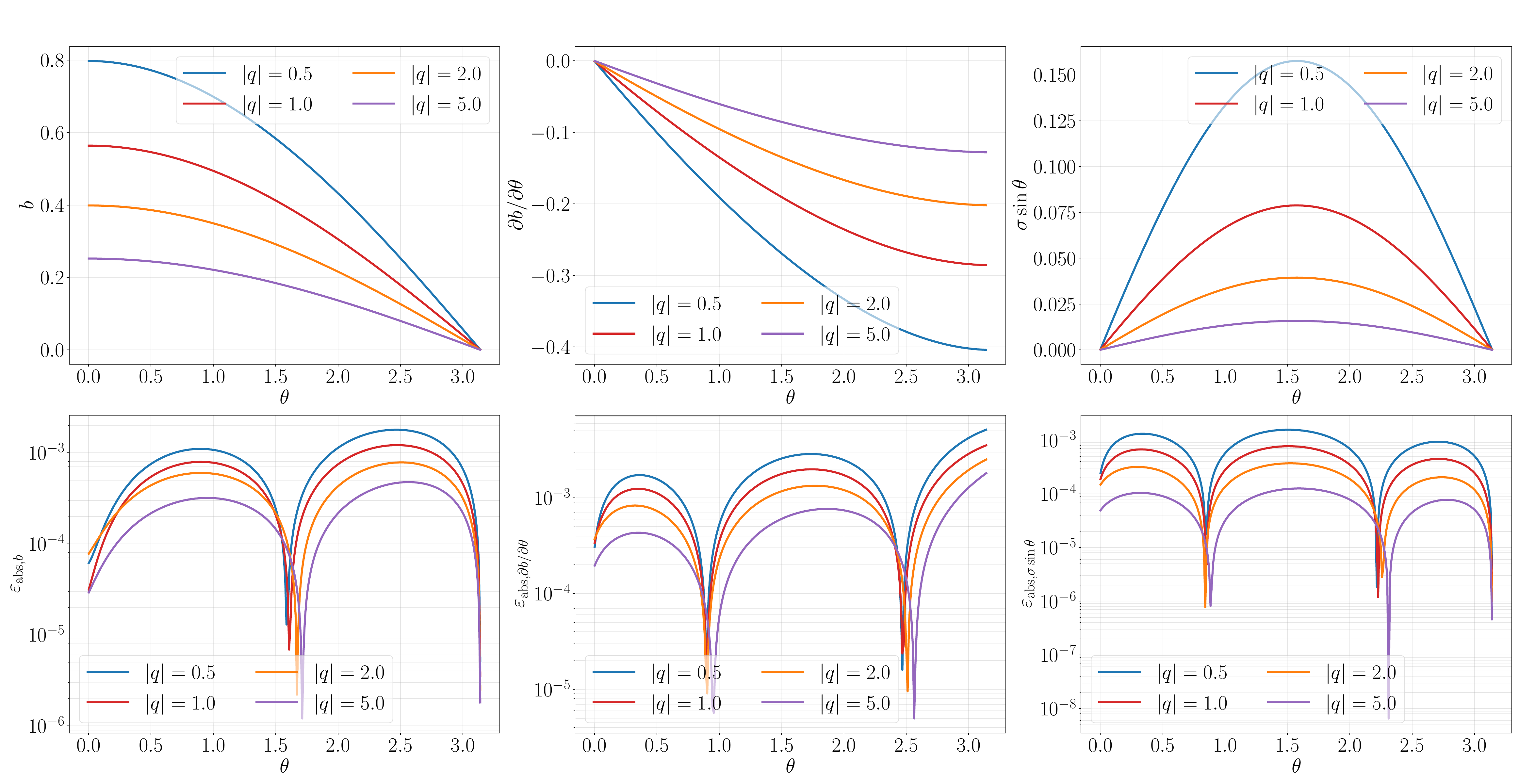}
        \label{fig:optsuppalpha0beta0}
    }

    \subfigure[$\alpha=0.384,\ \beta=1.3248$]{
        \centering
        \includegraphics[height=90mm,width=160mm]{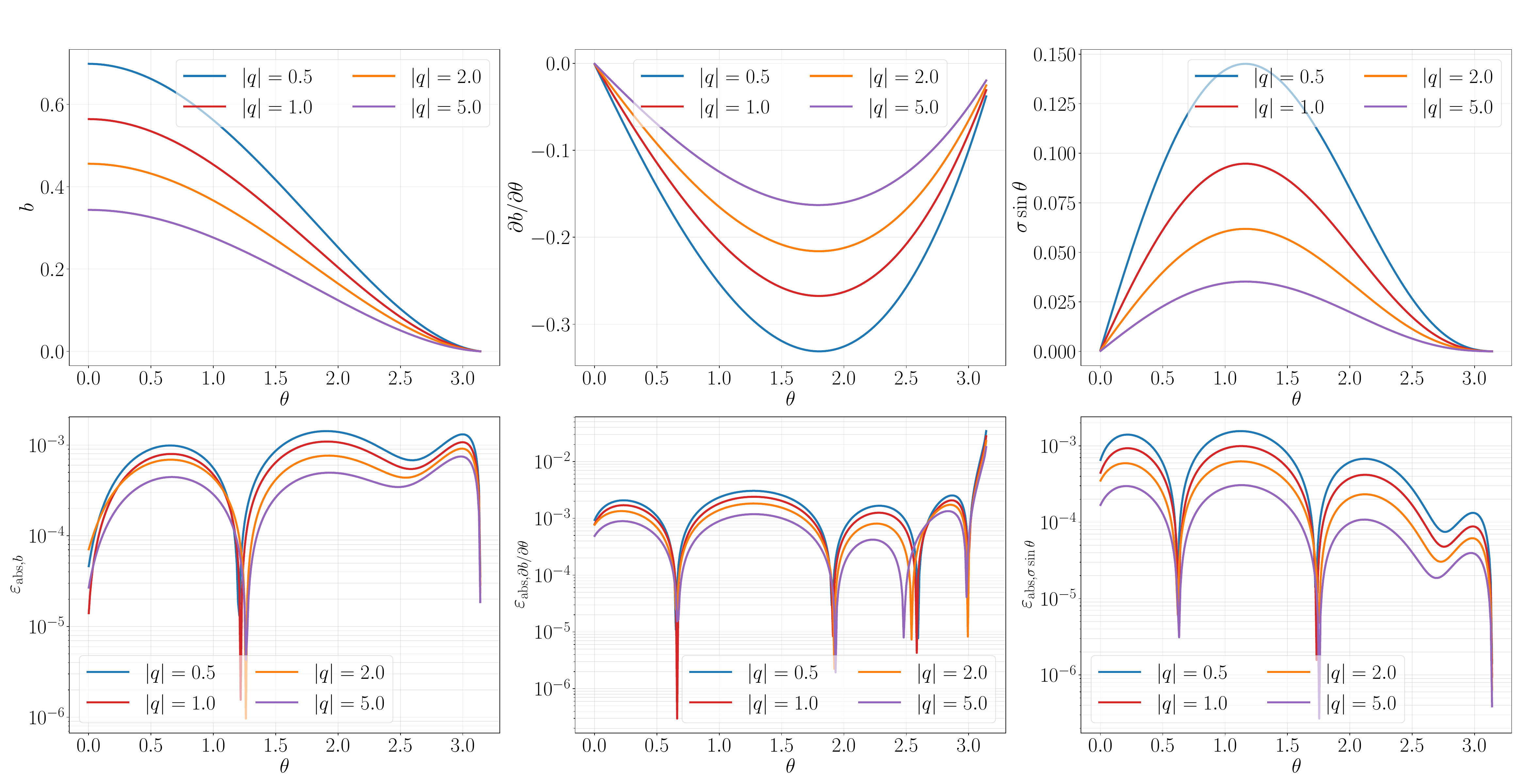}
        \label{fig:optsuppalpha0384beta13248}
    }
    \caption{
    Additional diagnostics for the Algorithm~\ref{algorithm:distributrans} in the two VSS test cases. In each subfigure, the first row compares the reconstructed and exact values of \(b(|q|,\theta)\), \(\partial_\theta b(|q|,\theta)\), and \(\sigma(|q|,\theta)\sin\theta\) at several representative relative speeds. The second row shows the corresponding absolute errors.}
    \label{fig:twooptimizsupple}
\end{figure}

\end{document}